\documentclass[11pt,reqno]{amsart}
\usepackage{amssymb}
\usepackage{amsmath}
\usepackage{amsthm}
\usepackage{graphicx}
\usepackage{caption}
\usepackage{bm}

\usepackage{booktabs}
\usepackage{adjustbox}
\usepackage{float}
\usepackage{subcaption}
\usepackage{multirow}
\usepackage[fulladjust]{marginnote}

\usepackage[colorlinks=true,urlcolor=blue,
citecolor=red,linkcolor=blue,linktocpage,pdfpagelabels,
bookmarksnumbered,bookmarksopen]{hyperref}
\usepackage[titletoc,title]{appendix}
\numberwithin{equation}{section} 
\newcounter{cont}[section] 

\newtheorem{thm}[cont]{Theorem}
\newtheorem{prop}[cont]{Proposition}
\newtheorem{lem}[cont]{Lemma}

\theoremstyle{definition}
\newtheorem{defn}[cont]{Definition}
 \theoremstyle{remark}
 \newtheorem{rem}[cont]{Remark}

\newcommand{\R}{\mathbb{R}}
\newcommand{\e}{\varepsilon}

\begin{document}

\title[Layer dynamics for Allen-Cahn equation with nonlinear diffusion]{Layer dynamics for 
the Allen--Cahn equation with nonlinear phase-dependent diffusion}

\author[J. A. Butanda Mej\'ia]{Jos\'e A. Butanda Mej\'ia}
\address[J. A. Butanda Mej\'ia]{Departamento de Matem\'aticas y Mec\'anica\\Instituto de Investigaciones en Matem\'aticas Aplicadas y en 
Sistemas\\Universidad Nacional Aut\'{o}noma de M\'{e}xico\\Circuito Escolar s/n, Ciudad Universitaria, C.P. 04510\\Cd. de M\'{e}xico (Mexico)}
\email{alex.87butanda@gmail.com}

\author[D. Casta\~{n}\'{o}n Quiroz]{Daniel Casta\~{n}\'{o}n Quiroz}
\address[D. Casta\~{n}\'{o}n Quiroz]{Departamento de Matem\'aticas y Mec\'anica\\Instituto de Investigaciones en Matem\'aticas Aplicadas y en 
Sistemas\\Universidad Nacional Aut\'{o}noma de M\'{e}xico\\Circuito Escolar s/n, Ciudad Universitaria, C.P. 04510\\Cd. de M\'{e}xico (Mexico)}
\email{daniel.castanon@iimas.unam.mx}

\author[R. Folino]{Raffaele Folino}
\address[R. Folino]{Departamento de Matem\'aticas y Mec\'anica\\Instituto de Investigaciones en Matem\'aticas Aplicadas y en 
Sistemas\\Universidad Nacional Aut\'{o}noma de M\'{e}xico\\Circuito Escolar s/n, Ciudad Universitaria, C.P. 04510\\Cd. de M\'{e}xico (Mexico)}
\email{folino@aries.iimas.unam.mx}

\author[L. F. L\'{o}pez R\'{\i}os]{Luis F. L\'{o}pez R\'{\i}os}
\address[L. F. L\'{o}pez R\'{\i}os]{Departamento de Matem\'aticas y Mec\'anica\\Instituto de Investigaciones en Matem\'aticas Aplicadas y 
en Sistemas\\Universidad Nacional Aut\'{o}noma de M\'{e}xico\\Circuito Escolar s/n, Ciudad Universitaria, C.P. 04510\\Cd. de M\'{e}xico (Mexico)}
\email{luis.lopez@aries.iimas.unam.mx}
 

\maketitle

\begin{abstract} 
The goal of this paper is to describe the metastable dynamics of the solutions to the reaction-diffusion equation with nonlinear phase-dependent diffusion
$u_t=\e^2(D(u)u_x)_x-f(u)$, where $D$ is a strictly positive function and $f$ is a bistable reaction term.
We derive a system of ordinary differential equations describing the slow evolution of the \emph{metastable states}, whose existence has been proved in \cite{FHLP}.
Such a system generalizes the one derived in the pioneering work \cite{Carr-Pego} to describe the metastable dynamics for the classical Allen--Cahn equation, 
which corresponds to the particular case $D\equiv1$.
\end{abstract}

\section{Introduction}\label{sec:intro}
\textbf{Metastability} is a broad term used in the analysis of partial differential equations to describe the persistence of unsteady structures for a very long time;
\emph{metastable dynamics} appears when some generic solutions of a time-dependent equation seem
motionless while still evolving very slowly and then eventually, after a very long time, change dramatically and reach a stable configuration. 
The seminal example of PDE whose solutions exhibit metastability is the one-dimensional Allen--Cahn equation 
\begin{equation}\label{eq:cAC}
	u_t = \e^2 u_{xx} - f(u), \quad x \in (a,b), \: t > 0,
\end{equation}
introduced to model the motion of spatially non-uniform phase structures in crystalline solids, see \cite{Allen-Cahn}. 
It describes the state of a system confined in a bounded space domain in terms of a scalar phase field, $u = u(x,t)$,
depending on space and time variables, which interpolates two homogeneous pure components, 
$u = \alpha$ and $u = \beta$, of the binary alloy. 
The balanced bistable reaction term $f$ can be seen as the derivative of a potential $F\in C^3(\R)$,
which is a prescribed function of the phase field $u$, having a double-well shape with local minima at the preferred $\alpha$- and $\beta$-phases. 
The metastable dynamics of the solutions to \eqref{eq:cAC} was analyzed and described in unpublished notes of John Neu 
and the landmark works \cite{Bron-Kohn,Carr-Pego,Fusco-Hale};
since then, there are many papers investigating this very interesting phenomenon, see for instance \cite{Bet-Sme,Chen,Otto-Rez,Ward,West} and references therein.

The main goal of this paper is to describe in detail the metastable dynamics for the following Allen--Cahn equation with phase-dependent diffusion
\begin{equation}\label{eq:D-model}
	u_t=\e^2(D(u)u_x)_x-f(u), \qquad \qquad x\in  (a,b), \; t>0,
\end{equation}
endowed with homogeneous Neumann boundary conditions
\begin{equation}\label{eq:Neu}
	u_x(a,t)=u_x(b,t)=0, \quad \quad \; t>0,
\end{equation}
and initial datum
\begin{equation}\label{eq:initial}
	u(x,0)=u_0(x), \qquad \qquad x\in[a,b].
\end{equation}
In \eqref{eq:D-model}, $\e>0$ is a small parameter, the diffusion coefficient $D = D(u)$ is a strictly positive function of the phase field and 
the reaction term $f=f(u)$ is assumed to be of \emph{bistable} type. More precisely, we assume that there exists an open interval $I\subset\R$ such that $D\in C^1(I)$ satisfies 
\begin{equation}\label{eq:ass-D}
	D(u)\geq d>0, \qquad \qquad \forall\, u\in I,
\end{equation}
and $f\in C^2(I)$ is such that
\begin{equation}\label{eq:ass-f}
	f(\alpha)=f(\beta)=0, \qquad \qquad f'(\alpha)>0, \quad f'(\beta)>0,
\end{equation}
for some $\alpha<\beta$ with $[\alpha,\beta]\subset I$. 
Clearly, if we choose $D(u)\equiv 1$, then we recover the classical Allen--Cahn equation \eqref{eq:cAC}.
There are several physical situations where the mathematical modelling involves (strictly positive) phase-dependent diffusivity coefficient,
such as the experimental exponential diffusion function for metal alloys \cite{Wagner} and the Mullins diffusion model for thermal grooving \cite{Broa,Mullins}.
The first example pertains to the description of the physical properties of metal alloys, 
for which it is well-known that the characteristic lattice parameter varies with the metal composition \cite{HaCr49,Wagner}. 
This function can be computed from experimental results and in \cite{Wagner} the author claims that the function that best fits many experiments on binary alloys 
has the shape of an exponential function
\begin{equation}
	D = D(u) = D_0 \exp \left\{c_0\left(u - \tfrac{1}{2}(\alpha + \beta)\right)\right\},
\end{equation}
where $c_0>0$ and $D_0>0$ are experimental constants.
The second example refers to thermal grooving, that is, the development of surface groove profiles on a heated polycrystal 
by the mechanism of evaporation-condensation (cf. Mullins \cite{Mullins}, Broadbridge \cite{Broa}). 
The metal polycrystal is assumed to be in quasi-equilibrium of its vapor, and the surface diffusion process, 
as well as the mechanism of evaporation and condensation, are modelled via Gibbs--Thompson formula. 
Since the properties of the interface do not depend on its orientation, this is essentially a one-dimensional phenomenon. 
After an appropriate transformation (see \cite{Broa}), the Mullins nonlinear diffusion model of groove development 
can be expressed in terms of a one-dimensional diffusion equation where the nonlinear diffusion coefficient is given by
\begin{equation}\label{eq:MullinsD}
	D(u) = \frac{D_0}{1+ u^2},
\end{equation}
with $D_0$ a positive constant.

The crucial assumptions in this work concern the interaction between $f$ and $D$ and we suppose that
\begin{equation}\label{eq:ass-int0}
	\int_\alpha^\beta f(s)D(s)\,ds=0, 
\end{equation}	
and
\begin{equation}\label{eq:ass-int1}
	\int_\alpha^u f(s)D(s)\,ds>0, \qquad \forall\,u\neq\alpha,\beta.
\end{equation}
It is important to notice that \eqref{eq:ass-D}-\eqref{eq:ass-f}-\eqref{eq:ass-int0}-\eqref{eq:ass-int1} imply that the function
\begin{equation}\label{eq:G}
	G(u):=\int_{\alpha}^{u}f(s)D(s)\, ds
\end{equation}
is a double-well potential with wells of equal depth, i.e. $G:I\to\R$ satisfies 
\begin{equation}\label{eq:ass-G}
	\begin{aligned}
		G(\alpha)=G(\beta)=G'(\alpha)&=G'(\beta)=0, \qquad G''(\alpha)>0,\quad G''(\beta)>0, \\ 
		&G(u)>0, \quad \forall\, u\neq\alpha,\beta.
	\end{aligned}
\end{equation}
\begin{rem}
Equation \eqref{eq:D-model} is a parabolic equation in $I$ in view of the assumption \eqref{eq:ass-D} and 
throughout the paper we consider initial data satisfying
\begin{equation}\label{eq:u0continI}
	M_1\leq u_0(x)\leq M_2, \qquad \qquad \forall\,x\in[a,b],
\end{equation}
for some $M_1<M_2$ such that $[M_1,M_2]\subset I$.
Therefore, we can use the classical maximum principle to conclude that the solution $u$ to the initial boundary value problem
 \eqref{eq:D-model}-\eqref{eq:initial} remains in $I$ for all time $t\geq0$.
 As we shall see, assumption \eqref{eq:ass-D} plays a crucial role in our analysis, which is not valid if we allow the diffusion coefficient $D$ to vanish at one or more points.
 To the best of our knowledge, the study of metastable dynamics for \eqref{eq:D-model} in the \emph{degenerate} setting -  when $D$ can vanish at some points -
 remains an open problem. We refer the interested reader to \cite{Nodea}, in which a related but distinct model is examined, 
 wherein either $D$ or $f'$ may vanish at $\alpha$ and/or $\beta$.
\end{rem}

The model \eqref{eq:D-model} with $D,f$ satisfying \eqref{eq:ass-D}-\eqref{eq:ass-f} and \eqref{eq:ass-int0}-\eqref{eq:ass-int1} has been studied in \cite{FHLP},
where it is proved that some solutions exhibit metastable dynamics:
if the initial datum $u_0$ has a particular \emph{transition layer structure}, then the time-dependent solution maintains the same (unstable) structure for an exponentially long time,
namely a time of $\mathcal{O}(\exp(C/\e))$, where $C>0$ is a constant, see Theorem \ref{thm:main} below.
We recall that this phenomenon is a peculiarity of the one-dimensional model and that the multidimensional version of \eqref{eq:D-model} 
($x\in\Omega\subset\mathbb R^n$, with $n>1$) has been studied in \cite{EFHPS}, where the authors describe generation and propagation of interfaces moving 
with normal velocity proportional to its mean curvature. Hence, \cite{EFHPS} extends the classical results valid for the multidimensional Allen--Cahn equation
\cite{Chen2, demot-sch} to the case of a non-degenerate ($D$ is strictly positive) nonlinear diffusion.

The motivation for this article is twofold.
On the one hand, this paper is a continuation of \cite{FHLP}:
once the existence of metastable states for \eqref{eq:D-model}-\eqref{eq:Neu} has been established, the aim is to describe their slow dynamics in detail.
On the other hand, our analysis generalizes the result contained in \cite{Carr-Pego}, concerning the layer dynamics for the classical Allen--Cahn equation \eqref{eq:cAC}.
Indeed, we will derive a system of ordinary differential equations, which describes the exponentially slow motion of the metastable states for \eqref{eq:D-model}-\eqref{eq:Neu}
for generic functions $D,f$ satisfying \eqref{eq:ass-D}-\eqref{eq:ass-f} and \eqref{eq:ass-int0}-\eqref{eq:ass-int1}.
For completeness and to maintain the self-contained nature of this paper, we recall here the main result from \cite{FHLP}.
Let us fix $N\in\mathbb{N}$ and a {\it piecewise constant function} $w$ 
assuming only the values $\alpha,\beta$ and with exactly $N$ jumps located at $h_1,\dots,h_N$, namely 
\begin{equation}\label{eq:vstruct}
	w:[a,b]\rightarrow\{\alpha,\beta\}\  \hbox{with $N$ jumps located at } a<h_1<h_2<\cdots<h_N<b.
\end{equation}	
Moreover, we fix $r>0$ such that
\begin{equation}\label{eq:r}
	h_i+r<h_{i+1}-r, \ \hbox{ for}\ i=1,\dots,N, \qquad   a\leq h_1-r,\qquad h_N+r\leq b.
\end{equation}
Finally, let us define the energy
\begin{equation*}
	E_\e[u]:=\int_a^b\left\{\frac\e2[D(u)u_x]^2+\frac{G(u)}\e\right\}\,dx,
\end{equation*}
where $G$ is the function defined in \eqref{eq:G}, and the positive constant
\begin{equation*}
	\gamma_0:=\int_{\alpha}^{\beta}\sqrt{2G(s)}D(s)\, ds,
\end{equation*}
which represents the \emph{minimum energy} to have a transition between $\alpha$ and $\beta$, see \cite[Proposition 2.4]{FHLP}.
\begin{defn}
Let $D$ and $f$ satisfying \eqref{eq:ass-D}, \eqref{eq:ass-f} and \eqref{eq:ass-int0}-\eqref{eq:ass-int1},
and let $w$ be a piecewise constant function like in \eqref{eq:vstruct}-\eqref{eq:r}.
We say that a function $u^\e_0\in H^1(a,b)$ has an \emph{$N$-transition layer structure} if 
\begin{equation}\label{eq:ass-u0}
	\lim_{\varepsilon\rightarrow 0} \|u^\varepsilon_0-w\|_{{}_{L^1}}=0,
\end{equation}
and there exist $C>0$ and $A\in(0,r\sqrt{2\lambda})$, where 
\begin{equation}\label{eq:lambda}
	\lambda:=\min\left\{\frac{f'(\alpha)}{D(\alpha)},\frac{f'(\beta)}{D(\beta)}\right\},
\end{equation}
such that
\begin{equation}\label{eq:energy-ini}
	E_\varepsilon[u^\varepsilon_0]\leq N\gamma_0+C\exp(-A/\e),
\end{equation}
for any $\varepsilon\ll1$. 
\end{defn}
The main goal of \cite{FHLP} is to show that if the initial datum has an $N$-\emph{transition layer structure}, 
then the solution to the IBVP \eqref{eq:D-model}-\eqref{eq:Neu}-\eqref{eq:initial} maintains such a structure for an exponentially long time as $\e\to0^+$.
\begin{thm}[\cite{FHLP}]\label{thm:main}
Assume that $f,D\in C^2(I)$ satisfy \eqref{eq:ass-D}-\eqref{eq:ass-f}-\eqref{eq:ass-int0}-\eqref{eq:ass-int1}.
Let $w,r$ be as in \eqref{eq:vstruct}-\eqref{eq:r} and let $A\in(0,r\sqrt{2\lambda})$, with $\lambda$ defined in \eqref{eq:lambda}.
If $u^\varepsilon$ is the solution of \eqref{eq:D-model}-\eqref{eq:Neu}-\eqref{eq:initial} 
with initial datum $u_0^{\varepsilon}$ satisfying \eqref{eq:u0continI}, \eqref{eq:ass-u0} and \eqref{eq:energy-ini}, then, 
\begin{equation}\label{eq:limit}
	\sup_{0\leq t\leq m\exp(A/\varepsilon)}\|u^\varepsilon(\cdot,t)-w\|_{{}_{L^1}}\xrightarrow[\varepsilon\rightarrow0]{}0,
\end{equation}
for any $m>0$.
\end{thm}
Theorem \ref{thm:main} states that the solution maintains the same structure as the initial datum (at least) for an exponentially long time; 
 in this paper, we carry out a detailed asymptotic analysis of the metastable dynamics 
 and derive a reduced system of ODEs governing the motion of the transition layers $h_1, \dots, h_N$.
 Let us briefly present our main result:
assume that the initial datum in \eqref{eq:initial} has a $N$-transition layer structure and (for definiteness) assume that $u_0(a)$ is close to $\alpha$.
Then, the evolution of the layer positions is determined by the system of ODEs
\begin{equation}\label{eq:h_j-intro}
	h_j'=\frac{\e}{S_G}(\theta^{j+1}-\theta^j), \qquad \qquad j=1,\dots, N,
\end{equation}
where $h_j:=h_j(t)$, for $j=1,\dots,N$, are the time dependent functions describing the movement of the transition points, 
the positive constant $S_G$, which depends only on the function $G$ defined in \eqref{eq:G}, is given by
$$S_G:=\int_\alpha^\beta\sqrt{2G(u)}\,du,$$
and $\theta^j$ is a function depending on the distance $\ell_j:=h_j-h_{j-1}$, 
where we defined $h_0:=2a-h_1$ and $h_{N+1}:=2b-h_N$ in view of \eqref{eq:Neu}.
In particular, $\theta^j$ satisfies
\begin{equation}\label{eq:theta^j}
	\theta^{j}:=\left\{\begin{aligned}
		&\tfrac{1}{2}A_\alpha^2D(\alpha)^2K_\alpha^2\exp\left(-\frac{A_\alpha\ell_j}{\e}\right)\left(1+\mathcal{O}\left\{\e^{-1}\exp\left(-\frac{A_\alpha\ell_j}{2\e}\right)\right\}\right),
		 &j \textrm{ odd},\\
		&\tfrac{1}{2}A_\beta^2D(\beta)^2K_\beta^2\exp\left(-\frac{A_\beta\ell_j}{\e}\right)\left(1+\mathcal{O}\left\{\e^{-1}\exp\left(-\frac{A_\beta\ell_j}{2\e}\right)\right\}\right), 	
		\qquad &j \textrm{ even},
		\end{aligned}\right.
\end{equation}
where 
$$A_\alpha:=\frac{\sqrt{G''(\alpha)}}{D(\alpha)}, \qquad \qquad A_\beta:=\frac{\sqrt{G''(\beta)}}{D(\beta)},$$
$$K_\alpha:=(\beta-\alpha)\exp\left\{\int_{\alpha}^{\frac{\alpha+\beta}{2}}\left(\frac{A_\alpha D(t)}{\sqrt{2G(t)}}-\frac{1}{t-\alpha}\right)\,dt\right\},$$
and
$$K_\beta:=(\beta-\alpha)\exp\left\{\int_{\frac{\alpha+\beta}{2}}^{\beta}\left(\frac{A_\beta D(t)}{\sqrt{2G(t)}}-\frac{1}{\beta-t}\right)\,dt\right\}.$$
Equations \eqref{eq:h_j-intro} generalize the result of Carr and Pego \cite{Carr-Pego} 
because if we choose $D\equiv1$ and $f$ satisfying \eqref{eq:ass-f}-\eqref{eq:ass-int0}-\eqref{eq:ass-int1} with $\alpha=-1$ and $\beta=1$,
we recover their equations.
Hence, we are able to analyze the role played by the diffusion coefficient $D$ in the slow motion of the layers. 
The qualitative behavior of $h_j$ does not change and the analysis contained in \cite[Section 6]{Carr-Pego} is still valid for \eqref{eq:h_j-intro}.
However, the constant $S_G$ and the function $\theta^j$ depend on the function $D$ which, 
as a consequence, plays a crucial role on the evolution of the transition points $h_j$.

According to \eqref{eq:h_j-intro}, any point $h_j$ moves to the right or to the left depending on the sign of the difference $\theta^{j+1}-\theta^j$;
to be more precise, if $j$ is odd and $\e>0$ is very small, then $h_j$ moves to the right (left) if 
$$A_\beta\ell_{j+1}<A_\alpha\ell_j, \qquad \qquad (A_\beta\ell_{j+1}>A_\alpha\ell_j).$$
Similarly, if $j$ is even, then $h_j$ moves to the right (left) if 
$$A_\alpha\ell_{j+1}<A_\beta\ell_j, \qquad \qquad (A_\alpha\ell_{j+1}>A_\beta\ell_j).$$
As a consequence, the direction in which $h_j$ moves depends on $D$, because the constants $A_\alpha, A_\beta$ depend on $D$.
In the special case $A_\beta=A_\alpha$, the transition point $h_j$ is simply attracted by the closest layer ($h_{j-1}$ or $h_{j+1}$);
moreover, if the distance $\ell_i$ between two layers is much smaller than the others, then $\theta^i$ is much bigger than $\theta^j$, for any $j\neq i$ and 
system \eqref{eq:h_j-intro} depends essentially on a single ODE (for further details, see the analysis in \cite[Section 6]{Carr-Pego}, which is valid for the case $D\equiv1$). 
In this case, the two closest layers move towards each other with approximately the same speed and the other $N-2$ points are essentially static.
In the general case $A_\alpha\neq A_\beta$, we have a similar behavior but the condition to find the fastest points depends 
on the minimum between $A_\alpha\ell_{2j-1}$ and $A_\beta\ell_{2j}$.

Finally, notice that the speed of any point $h_j$ is exponentially small as $\e\to0^+$, being
$$\theta^j=\mathcal{O}\left\{\exp\left(-\frac{A\ell_j}{\e}\right)\right\}, \qquad \qquad \mbox{ for any } j=1,\dots,N,$$
where $A:=\min\{A_\alpha,A_\beta\}$.
Thus, also the speed of the points strongly depends on $D$ and, clearly, the dynamics is faster if we choose a diffusion coefficient $D$ such that $A\gg1$.
In Section \ref{sec:num}, we shall present some numerical solutions showing the slow evolution of the metastable states and validating the aforementioned behavior.

The rest of the paper is organized as follows.
We start our analysis in Section \ref{sec:stat} by proving the existence of particular stationary solutions for \eqref{eq:D-model}.
First, we consider stationary waves connecting the stable equilibria $\alpha,\beta$ and, as an additional result, we prove that 
if $D(\alpha)=D(\beta)=0$, then the stationary waves touch the two equilibria and as a consequence there exist steady states 
in the bounded interval $[a,b]$ with an arbitrary number of transitions if $\e>0$ is sufficiently small, see Proposition \ref{prop:comp}.
These solutions are known in literature as \emph{compactons}, see \cite{Dra-Rob,Nodea} and references therein.
Next, we introduce the periodic solutions that we use to construct the metastable states in $[a,b]$ and to define the functions $\theta^j$ appearing in \eqref{eq:h_j-intro}.
In Section \ref{sec:est-per}, we prove some estimates for the periodic solutions, that will be used to derive the system \eqref{eq:h_j-intro}.
In particular, the main result of Section \ref{sec:est-per} is Theorem \ref{thm:teta-lead}, where we prove the crucial expansion for $\theta^j$.
Section \ref{sec:layer} contains the derivation of the system \eqref{eq:h_j-intro}, that is the main result of the paper.
Following \cite{Carr-Pego}, we use the periodic solutions introduced in Section \ref{sec:stat} to approximate a metastable state $u^{\bm h(t)}$, 
with $N$ layers located at $\bm h:=(h_1,\dots, h_N)$. 
As in \cite{Carr-Pego}, we write the solution in the form $u=u^{\bm h}+v$, 
where the remainder $v$ is orthogonal in $L^2$ to appropriate {\it tangent vectors} $\tau^{\bm h}_j$. 
The particular choice of the functions $\tau^{\bm h}_j$, depending on the diffusion coefficient $D$, allows us to derive system \eqref{eq:h_j-intro}
and to generalize the result of \cite{Carr-Pego}. 
Finally, in Section \ref{sec:num} we present some numerical solutions illustrating the metastable dynamics.
 
\section{The stationary problem}\label{sec:stat}
In this section we prove existence of some particular solutions to the stationary problem associated to \eqref{eq:D-model}, namely the ordinary differential equation
\begin{equation}\label{eq:stat}
	\e^2(D(\varphi)\varphi')'-f(\varphi)=0,
\end{equation}
where $\varphi:=\varphi(x)$ and $'=\frac{d}{dx}$.
Multiplying equation \eqref{eq:stat} by $D(\varphi)\varphi'$ and using the definition \eqref{eq:G}, we obtain
$$\left[\frac{\e^2}{2}\left(D(\varphi)\varphi'\right)^2-G(\varphi)\right]'=0,$$
and, as a trivial consequence, 
\begin{equation}\label{eq:firstorder}
	\frac{\e^2}{2}\left(D(\varphi)\varphi'\right)^2=G(\varphi)+C,
\end{equation}
for some appropriate constant $C\in\R$.
\subsection{Stationary waves}
In particular, by choosing $C=0$ we end up with two constant solutions, $\varphi\equiv\alpha$ and $\varphi\equiv\beta$, and the following heteroclinic orbits, 
defined in the whole real line.
\begin{lem}\label{lem:stand}
Consider equation \eqref{eq:stat} with $D,f$ satisfying \eqref{eq:ass-D}-\eqref{eq:ass-f} and \eqref{eq:ass-int0}-\eqref{eq:ass-int1}.
Then, there exists a unique increasing solution $\Phi$ to \eqref{eq:stat}, subject to
$$\Phi(0)=\frac{\alpha+\beta}{2}, \qquad \lim_{x\to-\infty}\Phi(x)=\alpha, \qquad  \lim_{x\to\infty}\Phi(x)=\beta.$$
Moreover, $\Phi$ satisfies the following {\it exponential decay}:
\begin{equation*}
	\begin{aligned}
		&\Phi(x)-\alpha\leq c_1 e^{c_2x}, &\mbox{as } x\to-\infty, \\
		&\beta-\Phi(x)\leq c_1 e^{-c_2x}, \qquad &\mbox{as } x\to+\infty,
	\end{aligned}
\end{equation*}
for some $c_1,c_2>0$.
Similarly, there exists a unique decreasing solution $\Psi$ to \eqref{eq:stat}, subject to
$$\Psi(0)=\frac{\alpha+\beta}{2}, \qquad \lim_{x\to-\infty}\Phi(x)=\beta, \qquad  \lim_{x\to\infty}\Phi(x)=\alpha,$$
and $\Psi$ satisfies 
\begin{equation*}
	\begin{aligned}
		&\beta-\Phi(x)\leq c_1 e^{c_2x}, &\mbox{as } x\to-\infty, \\
		&\Phi(x)-\alpha\leq c_1 e^{-c_2x}, \qquad &\mbox{as } x\to+\infty,
	\end{aligned}
\end{equation*}
for some $c_1,c_2>0$.
\end{lem} 
\begin{proof}
First, let us choose $C=0$ in \eqref{eq:firstorder} and consider the Cauchy problem
\begin{equation*}
	\begin{cases}
		\e D(\Phi)\Phi'=\sqrt{2G(\Phi)},\\
		\Phi(0)=\frac{\alpha+\beta}{2}.
	\end{cases}
\end{equation*}
By using the standard separation of variable, it follows that the solution $\Phi$ is implicitly defined by
$$\int_\frac{\alpha+\beta}{2}^{\Phi(x)}\frac{D(s)}{\sqrt{2G(s)}}\,ds=\frac{x}{\e},$$
and the existence of the increasing solution is a consequence of \eqref{eq:ass-D}-\eqref{eq:ass-G} because 
$$G(s)\sim(s-\alpha)^2, \quad s\to\alpha \qquad \mbox{ and } \qquad G(s)\sim(s-\beta)^2, \quad s\to\beta.$$
Similarly, the existence of the decreasing solution can be proved by considering the Cauchy problem
\begin{equation*}
	\begin{cases}
		\e D(\Psi)\Psi'=-\sqrt{2G(\Psi)},\\
		\Psi(0)=\frac{\alpha+\beta}{2},
	\end{cases}
\end{equation*}
and the proof is complete.
\end{proof}
It has to be noticed that each assumption on $D,f$ plays a crucial role in Lemma \ref{lem:stand}.
Let us, for instance, consider the case when $D$ is positive but vanishes at $u=\alpha,\beta$: 
for simplicity's sake, assume
\begin{equation}\label{eq:D-deg}
	D(s)\sim(s-\alpha)^m, \mbox{ as } s\to\alpha^+, \qquad D(s)\sim(\beta-s)^m, \mbox{ as } s\to\beta^-,
\end{equation}
for some $m\geq1$.
In this case, $G''(\alpha)=G''(\beta)=0$ and 
$$G(s)\sim(s-\alpha)^{m+2}, \quad s\to\alpha^+, \qquad \quad \qquad G(s)\sim(\beta-s)^{m+2}, \quad s\to\beta^-.$$
As a consequence, one has
\begin{align}
	\int_\frac{\alpha+\beta}{2}^{\beta}\frac{D(s)}{\sqrt{2G(s)}}\,ds\sim\int_\frac{\alpha+\beta}{2}^{\beta}(\beta-s)^{\frac{m}{2}-1}<\infty,\label{eq:touch-beta}\\
	\int_{\alpha}^\frac{\alpha+\beta}{2}\frac{D(s)}{\sqrt{2G(s)}}\,ds\sim\int_{\alpha}^\frac{\alpha+\beta}{2}(s-\alpha)^{\frac{m}{2}-1}<\infty, \label{eq:touch-alpha}
\end{align}
for any $m\geq1$.
The finiteness of the two latter integrals implies the following result.
\begin{lem}\label{lem:touch}
Suppose that $D(u)>0$ for any $u\neq\alpha,\beta$ and that $D,f$ satisfy \eqref{eq:ass-f}, \eqref{eq:ass-int0}-\eqref{eq:ass-int1} and \eqref{eq:D-deg}. 
Then, equation \eqref{eq:stat} admits a solution $\Phi$ satisfying $\Phi(0)=\frac{\alpha+\beta}{2}$,
\begin{equation}\label{ugualea1e-1}
	\Phi(x)=\alpha, \quad x\leq x^\e_1, \qquad  \Phi(x)=\beta, \quad x\geq x^\e_2, \qquad \Phi'(x)>0, \quad x\in(x^\e_1,x^\e_2),
\end{equation}
where 
\begin{equation}\label{eq:x-eps}
	x^\e_1=-\e \bar x_1, \qquad \qquad x^\e_2=\e \bar x_2,
\end{equation}
for some $\bar x_i>0$, for $i=1,2$ which do not depend on $\e$.

Similarly, equation \eqref{eq:stat} admits a solution $\Psi$ satisfying $\Psi(0)=\frac{\alpha+\beta}{2}$,
\begin{equation*}
	\Psi(x)=\beta, \quad x\leq -x^\e_2, \qquad  \Psi(x)=\alpha, \quad x\geq -x^\e_1, \qquad \Psi'(x)<0, \quad x\in(-x^\e_2,-x^\e_1),
\end{equation*}
with $x^\e_i$ as in \eqref{eq:x-eps}.
\end{lem}
\begin{proof}
As in Lemma \ref{lem:stand}, choosing $C=0$ in \eqref{eq:firstorder}, we obtain the Cauchy problem
\begin{equation*}
	\begin{cases}
		\e D(\Phi)\Phi'=\sqrt{2G(\Phi)},\\
		\Phi(0)=\frac{\alpha+\beta}{2},
	\end{cases}
\end{equation*}
whose solution is implicitly given by 
$$\int_\frac{\alpha+\beta}{2}^{\Phi(x)}\frac{D(s)}{\sqrt{2G(s)}}\,ds=\frac{x}{\e}.$$
In view of \eqref{eq:touch-beta}, we deduce that $\Phi(x^\e_2)=\beta$ with
$$x^\e_2=\e\int_\frac{\alpha+\beta}{2}^{\beta}\frac{D(s)}{\sqrt{2G(s)}}\,ds.$$
On the other hand, from \eqref{eq:touch-alpha} it follows that $\Phi(x^\e_1)=\alpha$ with
$$x^\e_1=-\e\int_{\alpha}^\frac{\alpha+\beta}{2}\frac{D(s)}{\sqrt{2G(s)}}\,ds.$$
Since the solution reaches the values $\alpha$ and $\beta$, we obtain a solution of \eqref{eq:stat} defined in the whole real line, which satisfies \eqref{ugualea1e-1}-\eqref{eq:x-eps}.
The proof of the existence of the non-increasing function $\Psi$ is very similar.
\end{proof}
\begin{rem}
Notice that the solutions $\Phi,\Psi$ constructed in Lemma \ref{lem:touch} are $C^1(\R)$ if and only if $m\in[1,2)$ in \eqref{eq:D-deg}.
Indeed, in the non-decreasing case, the derivative satisfies $\Phi'(x)=0$, for any $x\leq x^\e_1$ and $x\geq x^\e_2$, while
$$\e \Phi'=\frac{\sqrt{2G(\Phi)}}{D(\Phi)}, \qquad \qquad \mbox{ in } (x^\e_1,x^\e_2).$$
In particular, 
$$\lim_{x\to (x^\e_2)^-}\Phi'(x)=\frac{\sqrt2}{\e}\lim_{x\to (x^\e_2)^-}(\beta-\Phi(x))^{1-\frac{m}2}=0 \qquad \mbox{ if and only if } \qquad m\in[1,2).$$
\end{rem}

\subsection{Existence of compactons in the degenerate case}

Thanks to Lemma \ref{lem:touch}, we can construct special stationary solutions, called \emph{compactons} (see \cite{Nodea} and references therein), 
to the boundary value problem \eqref{eq:D-model}-\eqref{eq:Neu}.

\begin{prop}\label{prop:comp}
Suppose that $D(u)>0$ for any $u\neq\alpha,\beta$ and that $D,f$ satisfy \eqref{eq:ass-f}, \eqref{eq:ass-int0}-\eqref{eq:ass-int1} and \eqref{eq:D-deg}. 
Let $N\in \mathbb{N}$ and let $h_1,h_2,\dots, h_N$ be any $N$ real numbers such that $a<h_1<h_2<\cdots <h_N<b$. 
Then, for any $\e\in(0,\bar\e)$ with $\bar\varepsilon>0$ sufficiently small, there exist two solutions $\varphi_1$ and $\varphi_2$ to \eqref{eq:D-model}-\eqref{eq:Neu} satisfying 
$\varphi_1(a)=\alpha$ and $\varphi_2(a)=\beta$, respectively, oscillating between $\alpha$ and $\beta$ and satisfying
\begin{equation}\label{eq:zeros}
	\varphi_i(x)=\frac{\alpha+\beta}{2}, \qquad \mbox{ if and only if } \qquad x\in\left\{h_1,h_2,\dots, h_N\right\}.
\end{equation} 
\end{prop}
\begin{proof}
We only prove the existence of the solution $\varphi_1$, being the proof of the existence of $\varphi_2$ very similar.
To this aim, we consider the function 
$$\Phi_1(x):=\Phi(x-h_1), \qquad \,  x\in\R,$$ 
where $\Phi$ is the non-decreasing solution of Lemma \ref{lem:touch}. 
Then, $\Phi_1(h_1)=\frac{\alpha+\beta}{2}$ and from \eqref{ugualea1e-1} it follows that 
$$\Phi_1(x)=\alpha, \quad  \mbox{ if } \quad x\in(-\infty, h_1+x_1^\e], \qquad \qquad \Phi_1(x)=\beta, \quad  \mbox{ if } \quad x\in [h_1+x_2^\e, +\infty).$$
Since $x_1^\e$ and $x_2^\e$ satisfy \eqref{eq:x-eps}, we can choose $\e>0$ small enough so that $h_1+x_1^\e >a$; 
as a consequence, the restriction of $\Phi_1$ to the interval $[a,h_1+x_2^\e]$, denoted with the same symbol, turns out to be equal to $\alpha$ for every $x\in [a,h_1+x_1^\e]$
and touching $\beta$ at $h_1+x_2^\e$. 
Hence, we have constructed the first transition located at $h_1$.
In order to construct the second transition, let us assume that $\e>0$ is so small that $h_1+x_2^\e\leq h_2-x_2^\e$ and define
$$\Phi_2(x):=\Psi(x-h_2), \qquad \,  x\in\R,$$ 
where $\Psi$ is the non-increasing solution of Lemma \ref{lem:touch}. 
Again, the restriction of the function $\Phi_{2}$ to the interval $[h_1+x_2^\e, h_{2}-x_1^\e ]$ (still using the same notation), satisfies  
$$\Phi_2(x)=\beta, \quad  \mbox{ if } \quad x\in[h_1+x_2^\e, h_2-x_2^\e], \qquad \Phi_2(h_2)=\frac{\alpha+\beta}{2},  \qquad \Phi_2(h_2-x_1^\e)=\alpha.$$
Hence, by defining
$$\varphi_1(x)=\begin{cases}
\Phi_1(x), \qquad x\in[a,h_1+x_2^\e],\\
\Phi_2(x), \qquad x\in[h_1+x_2^\e,h_2-x_1^\e],
\end{cases}$$
we constructed the desired solution in the interval $[a,h_2-x_1^\e]$ with two transitions located at $h_1,h_2$.
By proceeding in the same way and alternatively using $\Phi$ and $\Psi$, one can construct a solution in the interval $[a,b]$ satisfying \eqref{eq:zeros}.
\end{proof}

We, again, point out that the existence of the compactons of Proposition \ref{prop:comp} is a consequence 
of the boundedness of the integrals in \eqref{eq:touch-beta}-\eqref{eq:touch-alpha}.
As we proved in Lemma \ref{lem:stand}, if $D,f$ satisfy \eqref{eq:ass-D}-\eqref{eq:ass-f} and \eqref{eq:ass-int0}-\eqref{eq:ass-int1},
the (strictly) monotone profiles $\Phi$ and $\Psi$ never touches $\alpha$ and $\beta$ and it is not possible to construct compactons.
\subsection{Periodic solutions}
We conclude this section by proving the existence of solutions to the following problems:
given $\ell>0$, let $\varphi(\cdot,\ell,+1)$ be the solution to 
\begin{equation}\label{eq:per-l}
	\e^2(D(\varphi)\varphi_x)_x-f(\varphi)=0, \qquad \quad
	\varphi\bigl(-\tfrac12\ell\bigr)=\varphi\bigl(\tfrac12\ell\bigr)=\frac{\alpha+\beta}{2},
\end{equation}
with $\varphi>\frac{\alpha+\beta}{2}$ in $(-\tfrac12\ell,\tfrac12\ell)$, and let $\varphi(\cdot,\ell,-1)$ be the solution to \eqref{eq:per-l} with $\varphi<\frac{\alpha+\beta}{2}$ in $(-\tfrac12\ell,\tfrac12\ell)$. 
These special solutions shall play a crucial role in Section \ref{sec:layer}, where we study the layer dynamics for the metastable solutions of \eqref{eq:D-model}-\eqref{eq:Neu}.
\begin{prop}\label{prop:ell}
Consider problem \eqref{eq:per-l} with $D,f$ satisfying \eqref{eq:ass-D}-\eqref{eq:ass-f} and \eqref{eq:ass-int0}-\eqref{eq:ass-int1}.
Then, there exists $\rho_0>0$ such that if $\e/\ell<\rho_0$, then the functions $\varphi(\cdot,\ell,\pm1)$ are well-defined, they are smooth and depend on $\e$ and $\ell$
only through the ratio $\e/\ell$.
\end{prop}
\begin{proof}
Let us prove that there exists a unique solution $\varphi(\cdot):=\varphi(\cdot,\ell,-1)$ of the problem \eqref{eq:per-l} satisfying the additional condition $\varphi<\frac{\alpha+\beta}{2}$ in $(-\tfrac12\ell,\tfrac12\ell)$; the proof for $\varphi(\cdot,\ell,+1)$ is similar. 
Proceeding as was done to obtain \eqref{eq:firstorder} and using the boundary conditions in \eqref{eq:per-l}, we end up with
\begin{equation}\label{eq:phipm-first}
	\e^2\left[D(\varphi)\varphi_x\right]^2=2\left[G(\varphi)-G(\varphi(0))\right].
\end{equation}
Hence, by using separation of variables and integrating in $[0,\ell/2]$, we conclude that any solution to \eqref{eq:per-l} satisfies 
\begin{equation}\label{eq:int-l-r}
	\sqrt2\int_{\varphi(0)}^{\frac{\alpha+\beta}{2}}\frac{D(s)}{\sqrt{G(s)-G(\varphi(0))}}\,ds=\ell\e^{-1}=:r^{-1}.
\end{equation}
The solution we are looking for exists if and only if there exists $\varphi(0)<\frac{\alpha+\beta}{2}$ such that \eqref{eq:int-l-r} holds true.
It is worthy to mention that the improper integral in \eqref{eq:int-l-r} is finite if $\varphi(0)$ is sufficiently close to $\alpha$;
indeed, the assumption on $D,f$ \eqref{eq:ass-D}-\eqref{eq:ass-f}-\eqref{eq:ass-int0}-\eqref{eq:ass-int1} ensure that $G$ satisfies \eqref{eq:ass-G} and, 
as a consequence, $G'(\varphi(0))>0$ if $\varphi(0)$ is sufficiently close to $\alpha$, namely
$$\int_{\varphi(0)}^{\frac{\alpha+\beta}{2}}\frac{D(s)}{\sqrt{G(s)-G(\varphi(0))}}\,ds\sim\int_{\varphi(0)}^{\frac{\alpha+\beta}{2}}\frac{D(s)}{\sqrt{G'(\varphi(0))(s-\varphi(0))}}\,ds<\infty.$$
We shall give more details about equation \eqref{eq:int-l-r} in the next section; here, it is sufficient to notice that the integral in \eqref{eq:int-l-r} is a monotone function of $\varphi(0)$ 
for $\varphi(0)$ close to $\alpha$ and that 
$$\lim_{\varphi(0)\to\alpha^+}\sqrt2\int_{\varphi(0)}^{\frac{\alpha+\beta}{2}}\frac{D(s)}{\sqrt{G(s)-G(\varphi(0))}}\,ds=
\sqrt2\int_{\alpha}^{\frac{\alpha+\beta}{2}}\frac{D(s)}{\sqrt{G(s)}}\,ds=+\infty,$$
because of \eqref{eq:ass-D} and \eqref{eq:ass-G}.
Therefore, if $r>0$ is sufficiently small, then there exists a unique $\varphi(0)$ (depending on $r$ and close to $\alpha$) such that \eqref{eq:int-l-r} is satisfied.
\end{proof}

\section{Estimates for the $\ell$-periodic solutions}\label{sec:est-per}
In this section, we prove some useful properties of the solutions $\varphi(x,\ell,\pm1)$, whose existence has been proved in Proposition \ref{prop:ell}.
To be more precise, we establish some important expansions as $\e\to0^+$, which play a crucial role in the following section. 
Since $\varphi(0,\ell,\pm1)$ depends only on the ratio $r=\varepsilon/\ell$ (see Proposition \ref{prop:ell}), we can define
\begin{equation}\label{eq:teta-delta}
	\theta_\pm(r):=G(\varphi(0,\ell,\pm1)), \quad \delta_+(r):=\beta-\varphi(0,\ell,+1), \quad \delta_-(r):=\varphi(0,\ell,-1)-\alpha.
\end{equation}
As we saw in the proof of Proposition \ref{prop:ell}, $\varphi(0,\ell,-1)$ is close to $\alpha$ for $r>0$ small; similarly, $\varphi(0,\ell,+1)$ is close to $\beta$ and so
$\theta_\pm(r), \delta_\pm(r)$ goes to $0$ as $r\to0^+$. 
The main result of this section characterizes the leading terms in $\theta_\pm,\delta_\pm$ as $r\to 0$;
in particular, we show that  $\varphi(0,\ell,-1)$ is exponentially close to $\alpha$ and $\varphi(0,\ell,+1)$ is exponentially close to $\beta$ as $r\to0^+$.
Roughly speaking, the functions $\theta_\pm,\delta_\pm$ are exponentially small as $r\to0^+$.
\begin{thm}\label{thm:teta-lead}
Assume that $D,f$ satisfies \eqref{eq:ass-D}-\eqref{eq:ass-f}-\eqref{eq:ass-int0}-\eqref{eq:ass-int1}.
Then, there exists $r_0>0$ such that for $r\in(0,r_0)$
\begin{align*}
	\theta_-(r)&=\frac{1}{2}A_\alpha^2D(\alpha)^2K_\alpha^2\exp\left(-\frac{A_\alpha}{r}\right)\left(1+\mathcal{O}\left\{r^{-1}\exp\left(-\frac{A_\alpha}{2r}\right)\right\}\right), \\
	\theta_+(r)&=\frac{1}{2}A_\beta^2D(\beta)^2K_\beta^2\exp\left(-\frac{A_\beta}{r}\right)\left(1+\mathcal{O}\left\{r^{-1}\exp\left(-\frac{A_\beta}{2r}\right)\right\}\right),\\
	\delta_-(r)&=K_\alpha\exp\left(-\frac{A_\alpha}{2r}\right)\left(1+\mathcal{O}\left\{r^{-1}\exp\left(-\frac{A_\alpha}{2r}\right)\right\}\right), \\
	\delta_+(r)&=K_\beta\exp\left(-\frac{A_\beta}{2r}\right)\left(1+\mathcal{O}\left\{r^{-1}\exp\left(-\frac{A_\beta}{2r}\right)\right\}\right).
\end{align*}
where 
$$A_\gamma:=\frac{\sqrt{G''(\gamma)}}{D(\gamma)}, \qquad \qquad \mbox{ for } \quad \gamma=\alpha,\beta,$$
$$K_\alpha:=(\beta-\alpha)\exp\left\{\int_{\alpha}^{\frac{\alpha+\beta}{2}}\left(\frac{A_\alpha D(t)}{\sqrt{2G(t)}}-\frac{1}{t-\alpha}\right)\,dt\right\},$$
and
$$K_\beta:=(\beta-\alpha)\exp\left\{\int_{\frac{\alpha+\beta}{2}}^{\beta}\left(\frac{A_\beta D(t)}{\sqrt{2G(t)}}-\frac{1}{\beta-t}\right)\,dt\right\}.$$
\end{thm}
\begin{proof}
Let us prove the result only for $\theta_-$ and $\delta_-$; the proof for $\theta_+$ and $\delta_+$ is similar and therefore omitted.
Following \cite{Carr-Pego} and using Proposition \ref{prop:ell}, we can claim that the function $\varphi(\cdot,\ell,-1)$ is well defined and \eqref{eq:int-l-r} reads as
\begin{equation}\label{eq:r-z-}
	r^{-1}=\sqrt2\int_{z_-}^{\frac{\alpha+\beta}{2}}\frac{D(s)}{\sqrt{G(s)-\theta_-}}\,ds,
\end{equation}
where $G$ satisfies \eqref{eq:ass-G}, $\theta_-:=\theta_-(r)$ and $z_-:=z_-(r)=\alpha+\delta_-(r)$.
For small $t$, we can use the expansions 
\begin{align*}
	D(z_-+t)&=D(z_-)+\mathcal{O}(t),\\
	G(z_-+t)&=\theta_-+G''(z_-)\lambda t+\frac12G''(z_-)t^2+\mathcal{O}(t^3),
\end{align*}
where $\lambda:=G'(z_-)/G''(z_-)$.
Since $G'(\alpha)=0$ and $G''(\alpha)>0$ (see \eqref{eq:ass-G}), we can express $z_-$ as a smooth function of $\lambda$ for $z_-$ close to $\alpha$,
with $z_--\alpha=\lambda+\mathcal{O}(\lambda^2)$.
Hence, we deduce
\begin{align}
	\delta_-(r)&=\lambda+\mathcal{O}(\lambda^2), \label{eq:delta-lambda}\\
	\theta_-(r)&=G(\alpha+\delta_-(r))=G(\alpha+\lambda+\mathcal{O}(\lambda^2)). \label{eq:theta-lambda}
\end{align}
Moreover, \eqref{eq:r-z-} takes the form
\begin{equation}\label{eq:r-z-2}
	\frac{\sqrt{G''(z_-)}}{2D(z_-)r}=\int_{0}^{-z_-+\frac{\alpha+\beta}{2}}\frac{1+t\mu_1(t)}{\sqrt{t^2+2\lambda t+t^3\mu(t,\lambda)}}\,dt=:I(\lambda),
\end{equation}
where $\mu_1,\mu_2$ are smooth functions.
From \cite[Lemma 7.1]{Carr-Pego}, it follows that there exists $\lambda_0>0$ such that for $\lambda\in(0,\lambda_0)$ one has
$$I(\lambda)=(-\ln\lambda)(1+c_1\lambda+\dots c_n\lambda^n)+J(\lambda),$$
where $c_j$ are constants and $J$ is a $C^n$ function in $[0,\lambda_0]$.
As a consequence, \eqref{eq:r-z-2} becomes
\begin{equation}\label{eq:r-z-3}
	\frac{A_\alpha}{2r}=(-\ln\lambda)(1+d_1\lambda+\cdots + d_n\lambda^n)+J_1(\lambda),
\end{equation}
where $d_j$ are constants and $J_1$ is a $C^n$ function in $[0,\lambda_0]$.
Let us introduce the new parameter $R$ defined by
$$\lambda=\exp\left(-\frac{A_\alpha}{2r}+R\right),$$
so that we can rewrite equation \eqref{eq:r-z-3} as
$$R-J_1(0)-S(R,r)=0,$$
where
$$\frac{\partial^m S}{\partial r^m}=\mathcal{O}\left\{r^{-(2m+1)}\exp\left(-\frac{A_\alpha}{2r}\right)\right\}, \qquad \qquad m=0,1,\dots,n.$$
By using the implicit function theorem to solve the equation for $R$ and $r$, we deduce that there exists a $C^n$ function $R:=R(r)$ defined on $[0,r_0)$
given by 
$$R(r)=J_1(0)+\mathcal{O}\left\{r^{-1}\exp\left(-\frac{A_\alpha}{2r}\right)\right\}.$$
As a consequence, we end up with
\begin{equation}\label{eq:lambda-final}
	\lambda=\exp\left(-\frac{A_\alpha}{2r}\right)\left(K_\alpha+\mathcal{O}\left\{r^{-1}\exp\left(-\frac{A_\alpha}{2r}\right)\right\}\right),
\end{equation}
where $K_\alpha:=\exp\{J_1(0)\}$.
It remains to compute $K_\alpha$ and we have the formula for $\theta_-$: let us rewrite \eqref{eq:r-z-} as
\begin{equation*}
	\begin{aligned}
		\frac{\sqrt{G''(z_-)}}{2D(z_-)r}&=
		\int_{0}^{-z_-+\frac{\alpha+\beta}{2}}\left(\frac{D(z_-+t)\sqrt{G''(z_-)}}{\sqrt2D(z_-)\sqrt{G(z_-+t)-\theta_-}}-\frac{1}{\sqrt{t^2+2\lambda t}}\right)\,dt\\
		&\qquad \qquad +\int_{0}^{-z_-+\frac{\alpha+\beta}{2}}\frac{1}{\sqrt{t^2+2\lambda t}}\,dt\\
		&=M_1(r)+M_2(r).
	\end{aligned}
\end{equation*}
The second integral can be explicitly computed and it is equal to
$$M_2(r)=\ln\left(-z_-+\frac{\alpha+\beta}{2}+\lambda+\sqrt{\left(-z_-+\frac{\alpha+\beta}{2}\right)^2+\lambda(-2z_-+\alpha+\beta)}\right)-\ln\lambda.$$
Since
$$\frac{\sqrt{G''(z_-)}}{2D(z_-)r}+\ln\lambda=M_1(r)+M_2(r)+\ln\lambda,$$
and as $r\to0^+$, one has $z_-\to\alpha$, $\lambda\to0$, $\theta_-\to0$, passing to the limit we conclude that
$$J_1(0)=\int_{0}^{\frac{\beta-\alpha}{2}}\left(\frac{A_\alpha D(\alpha+t)}{\sqrt{2G(\alpha+t)}}-\frac{1}{t}\right)\,dt+\ln(\beta-\alpha).$$
Therefore,
$$K_\alpha=\exp\left\{J_1(0)\right\}=(\beta-\alpha)\exp\left\{\int_{0}^{\frac{\beta-\alpha}{2}}\left(\frac{A_\alpha D(\alpha+t)}{\sqrt{2G(\alpha+t)}}-\frac{1}{t}\right)\,dt\right\}.$$
Substituting \eqref{eq:lambda-final} in \eqref{eq:delta-lambda}, we obtain the expansion for $\delta_-$, namely
$$\delta_-(r)=K_\alpha\exp\left(-\frac{A_\alpha}{2r}\right)\left(1+\mathcal{O}\left\{r^{-1}\exp\left(-\frac{A_\alpha}{2r}\right)\right\}\right).$$
Finally, from \eqref{eq:ass-G} and \eqref{eq:theta-lambda} it follows that
$$\theta_-(r)=\frac{1}{2}G''(\alpha)K_\alpha^2\exp\left(-\frac{A_\alpha}{r}\right)\left(1+\mathcal{O}\left\{r^{-1}\exp\left(-\frac{A_\alpha}{2r}\right)\right\}\right),$$
and the proof is complete.
\end{proof}
The analysis contained in this paper relies fundamentally on the expansions for $\theta_\pm$ and $\delta_\pm$ in Theorem \ref{thm:teta-lead} and they will be used frequently. 
Another result that we use in Section \ref{sec:layer} concerns the behavior of the integral of the functions $D(\varphi)\varphi_x^2$ in the intervals $[-\ell/2,0]$ and $[0,\ell/2]$, respectively.
\begin{lem}
Under the assumptions of Proposition \ref{prop:ell}, it follows that
\begin{equation}\label{eq:int-Dphi'}
	\begin{aligned}
		\mathcal{I}^D(r)&:=\mathcal{I}^D_-(r)+\mathcal{I}^D_+(r)=\\
		&=\int_{-\ell/2}^0D(\varphi(x,\ell,-1))\varphi_x(x,\ell,-1)^2\,dx+\int_0^{\ell/2}D(\varphi(x,\ell,+1))\varphi_x(x,\ell,+1)^2\,dx\\
		&=\e^{-1}S_G+E(r),
	\end{aligned}
\end{equation}
where $r=\e/\ell$,
\begin{equation}\label{eq:S_G}
	S_G:=\int_\alpha^\beta\sqrt{2G(u)}\,du,
\end{equation}
and $|E|\leq C\e^{-1}\max\left\{\delta_+(r),\delta_-(r)\right\}$, for some $C>0$ independent on $\e$.
\end{lem}
\begin{proof}
Let us consider the first integral in \eqref{eq:int-Dphi'}; by taking advantage of \eqref{eq:phipm-first} and the fact that $\varphi_x(x,\ell,-1)\leq0$, for any $x\in[-\ell/2,0]$, we obtain 
$$\e D(\varphi(x,\ell,-1))\varphi_x(x,\ell,-1)=-\sqrt{2[G(\varphi(x,\ell,-1))-\theta_-(r)]}.$$
Hence,
\begin{align*}
	\e\mathcal{I}^D_-(r)&=\int_{\varphi(0,\ell,-1)}^{\frac{\alpha+\beta}{2}}\sqrt{2[G(u)-\theta_-(r)]}\,du\\
	&=\int_\alpha^{\frac{\alpha+\beta}{2}}\sqrt{2G(u)}\,du-\int_\alpha^{\varphi(0,\ell,-1)}\sqrt{2G(u)}\,du\\
	&\qquad+\int_{\varphi(0,\ell,-1)}^{\frac{\alpha+\beta}{2}}\left\{\sqrt{2[G(u)-\theta_-(r)]}-\sqrt{2G(u)}\right\}\,du\\
	&=\int_\alpha^{\frac{\alpha+\beta}{2}}\sqrt{2G(u)}\,du-\int_\alpha^{\varphi(0,\ell,-1)}\sqrt{2G(u)}\,du\\
	&\qquad-\sqrt2\int_{\varphi(0,\ell,-1)}^{\frac{\alpha+\beta}{2}}\frac{\theta_-}{\sqrt{G(u)}+\sqrt{G(u)-\theta_-(r)}}\,du,
\end{align*}
and, as a consequence, one has
$$\mathcal{I}^D_-(r)=\e^{-1}\int_\alpha^{\frac{\alpha+\beta}{2}}\sqrt{2G(u)}\,du+\e^{-1}E_-(r),$$
where
\begin{equation*}
		E_-(r):=-\int_\alpha^{\varphi(0,\ell,-1)}\sqrt{2G(u)}\,du-
		\sqrt2\int_{\varphi(0,\ell,-1)}^{\frac{\alpha+\beta}{2}}\frac{\theta_-(r)}{\sqrt{G(u)}+\sqrt{G(u)-\theta_-(r)}}\,du.
\end{equation*}
Similarly, for the second integral in \eqref{eq:int-Dphi'}, we deduce 
$$\mathcal{I}^D_+(r)=\e^{-1}\int_{\frac{\alpha+\beta}{2}}^\beta\sqrt{2G(u)}\,du+\e^{-1}E_+(r),$$
where
\begin{equation*}
	E_+(r):=-\int_{\varphi(0,\ell,+1)}^\beta\sqrt{2G(u)}\,du
		-\sqrt2\int_{\frac{\alpha+\beta}{2}}^{\varphi(0,\ell,+1)}\frac{\theta_+(r)}{\sqrt{G(u)}+\sqrt{G(u)-\theta_+(r)}}\,du.
\end{equation*}
By summing up, we obtain
$$\mathcal{I}^D_-(r)=\e^{-1}S_G+E(r),$$
where $S_G$ is defined in \eqref{eq:S_G} and $E(r)=\e^{-1}E_-(r)+\e^{-1}E_+(r)$.
From Theorem \ref{thm:teta-lead}, it follows that
\begin{align*}
	|E_-(r)|&\leq C\delta_-(r)\sqrt{\theta_-(r)}+C\sqrt{\theta_-(r)}\leq C\delta_-(r),\\
	|E_+(r)|&\leq C\delta_+(r)\sqrt{\theta_+(r)}+C\sqrt{\theta_+(r)}\leq C\delta_+(r),
\end{align*}
and the proof is complete.
\end{proof}
Notice that Theorem \ref{thm:teta-lead} implies
\begin{equation}\label{eq:E-int}
	|E(r)|\leq C\exp\left(-\frac{A}{r}\right), \quad \mbox{ for } r\in(0,r_0), \qquad \mbox{ with } A:=\min\{A_\alpha,A_\beta\}.
\end{equation}
We continue this section by establishing estimates for the derivative of the functions $\varphi(x,\ell,\pm1)$ with respect to $\ell$.
To start with we derive the following formula.
\begin{lem}\label{lem:varphi_l}
Under the assumptions of Proposition \ref{prop:ell}, it follows that 
\begin{equation*}
	2\varphi_\ell(x,\ell,\pm1)=-(\textrm{\emph{sign }} x)\varphi_x(x,\ell,\pm1)+2w(x,\ell,\pm1),
\end{equation*}
for $x\in[-\ell,\ell]$, where 
\begin{equation*}
	w(x,\ell,\pm1):=\e^{-1}\ell^{-2}\theta_\pm'(r)\varphi_x(|x|,\ell,\pm1)\int_{\ell/2}^{|x|}\frac{ds}{\left[D(\varphi(s,\ell,\pm1))\varphi_x(s,\ell,\pm1)\right]^2},
\end{equation*}
for $x\neq0$ and 
\begin{equation*}
	w(0,\ell,\pm1):=-\e^{-1}\ell^{-2}\theta_\pm'(r)\left[\varphi_{xx}(0,\ell,\pm1)D(\varphi(0,\ell,\pm1))^2\right]^{-1},
\end{equation*}
where $r=\e/\ell$ and $\theta_\pm'(r)=\frac{d}{dr}\theta_\pm(r)$.
\end{lem}
\begin{proof}
Consider $x>0$ and differentiate equation \eqref{eq:phipm-first} with respect to $\ell$ to obtain
\begin{equation*}
	\e^2 D(\varphi)\varphi_x \left[D'(\varphi)\varphi_\ell\varphi_x+D(\varphi)\varphi_{x\ell}\right] = f(\varphi)D(\varphi)\varphi_\ell-\frac{d}{d\ell}\theta_\pm(r),
\end{equation*}
where we used that $G'(\varphi)=f(\varphi)D(\varphi)$.
Using equation \eqref{eq:stat} and the definition $r=\e/\ell$, we infer
$$f(\varphi)D(\varphi)=\e^2D'(\varphi)\varphi_x^2D(\varphi)+\e^2D^2(\varphi)\varphi_{xx}, \qquad \qquad \frac{d}{d\ell}\theta_\pm(r)=-\theta_\pm'(r)\e\ell^{-2}.$$
Therefore,
\begin{equation*}
	\e^2D^2(\varphi)\left[\varphi_x\varphi_{x\ell}-\varphi_{xx}\varphi_\ell\right]=\theta_\pm'(r)\e\ell^{-2}.
\end{equation*}
We can rewrite this as
\begin{equation}\label{eq:to-phi_l}
	\frac{d}{dx}\left(\frac{\varphi_l}{\varphi_x}\right)=\frac{\theta_\pm'(r)}{\e\ell^2D^2(\varphi)\varphi_x^2}.
\end{equation}
Integrating \eqref{eq:to-phi_l} in $(\ell/2,x)$, we end up with
$$\frac{\varphi_\ell(x,\ell,\pm1)}{\varphi_x(x,\ell,\pm1)}-\frac{\varphi_\ell(\ell/2,\ell,\pm1)}{\varphi_x(\ell/2,\ell,\pm1)}=
\e^{-1}\ell^{-2}\theta_\pm'(r)\int_{\ell/2}^{x}\frac{ds}{\left[D(\varphi(s,\ell,\pm1))\varphi_x(s,\ell,\pm1)\right]^2}.$$
In order to compute $\varphi_\ell(l/2,\ell,\pm1)$, let us differentiate the identity 
$$\varphi(\ell/2,\ell,\pm1)=\frac{\alpha+\beta}{2}$$
with respect to $\ell$. 
The result is
$$\frac12\varphi_x(\ell/2,\ell,\pm1)+\varphi_\ell(\ell/2,\ell,\pm1)=0,$$
which implies
$$\frac{\varphi_\ell(\ell/2,\ell,\pm1)}{\varphi_x(\ell/2,\ell,\pm1)}=-\frac12,$$
and we obtain the formula for $\varphi_\ell$, when $x>0$.
To obtain the formula for $x<0$ we can either repeat the above procedure using the boundary condition  
$$\varphi(-\ell/2,\ell,\pm1)=\frac{\alpha+\beta}{2},$$
or use the fact that $\varphi_\ell(x)=\varphi_\ell(-x)$, which holds since $\varphi$ is even in $x$.
To derive the formula for $w(0,\ell,\pm1)$ notice that 
$$\lim_{x\to0^+}\int_{\ell/2}^{x}\frac{ds}{\left[D(\varphi(s,\ell,\pm1))\varphi_x(s,\ell,\pm1)\right]^2}=
\e^2\lim_{x\to0^+}\int_{\ell/2}^{x}\frac{ds}{\sqrt{G(\varphi(x,\ell,\pm1))-\theta_\pm(r)}}=-\infty.$$
As a consequence, by using L'Hopital's rule we get
\begin{align*}
	\lim_{x\to0^+}\varphi_x(x,\ell,\pm1)\int_{\ell/2}^{x}\frac{ds}{\left[D(\varphi(s,\ell,\pm1))\varphi_x(s,\ell,\pm1)\right]^2}&=
	-\lim_{x\to0^+}\frac{1}{\varphi_{xx}(x,\ell,\pm1)D(\varphi(x,\ell,\pm1))^2}\\
	&=-\frac{1}{\varphi_{xx}(0,\ell,\pm1)D(\varphi(0,\ell,\pm1))^2},
\end{align*}
and the proof is complete.
\end{proof}
We conclude this section by obtaining an upper bound for the function $w$. 
\begin{lem}\label{lem:est-w}
Let $w(x,\ell,\pm1)$ be given by Lemma \ref{lem:varphi_l}.
Then, there exists $r_0>0$ such that for any $r\in(0,r_0)$
\begin{equation*}
	|w(x,\ell,\pm1)|\leq 
	\begin{cases}
		C\e^{-1}\delta_\pm(r), \qquad &\mbox{ for } x\in[-\ell/2-\e,\ell/2+\e],\\
		C\e^{-1}\theta_\pm(r), \qquad &\mbox{ for } x\in[-\ell/2-\e,-\ell/2+\e]\\
		&\mbox{ and for } x\in[\ell/2-\e,\ell/2+\e].
	\end{cases}
\end{equation*}
\end{lem}
\begin{proof}
We will prove the result only for $w(x,\ell,-1)$; the proof for $w(x,\ell,+1)$ is very similar and we omit it.
First, we consider $w(0,\ell,-1)$; from the formula given by Lemma \ref{lem:varphi_l} it follows that
$$|w(0,\ell,-1)|\leq\e^{-1}\ell^{-2}|\theta'_-(r)||\e^{-2}f(\varphi(0,\ell,-1))D(\varphi(0,\ell,-1))|^{-1},$$
where we used \eqref{eq:stat} and the fact that $\varphi_x(0,\ell,-1)=0$.
By using the definition of $\theta_-$ and $\delta_-$ \eqref{eq:teta-delta}, we obtain
$$\theta'_-(r)=\frac{d}{dr}G(\varphi(0,\ell,-1))=f(\varphi(0,\ell,-1))D(\varphi(0,\ell,-1))\delta'_-(r).$$
Therefore,
$$|w(0,\ell,-1)|\leq\e\ell^{-2}|\delta'_-(r)|\leq C\e^{-1}\delta_-(r),$$
where we used that $|\delta'_-(r)|\leq Cr^{-2}\delta_-(r)$, see Theorem \ref{thm:teta-lead}.
Now, we claim that 
\begin{equation}\label{eq:claim-w}
	|w(x,\ell,-1)|\leq\frac{D(\varphi(0,\ell,-1))}{D(\varphi(x,\ell,-1))}|w(0,\ell,-1)|, \qquad \quad \mbox{ for } x\in[0,\ell/2-\e].
\end{equation}
To prove this, define $\tilde w(x,\ell,-1):=D(\varphi(x,\ell,-1))w(x,\ell,-1)$.
Multiplying by $D(\varphi)$ and differentiating the formula for $w$ given by Lemma \ref{lem:varphi_l}, we infer
\begin{align*}
	\tilde w_x(x,\ell,-1)=&\e^{-1}\ell^{-2}\theta_-'(r)\left[D(\varphi(x,\ell,-1))\varphi_x(x,\ell,-1)\right]_x\int_{\ell/2}^{x}\frac{ds}{\left[D(\varphi(s,\ell,-1))\varphi_x(s,\ell,-1)\right]^2}\\
	&\qquad +\frac{\e^{-1}\ell^{-2}\theta_-'(r)}{D(\varphi_x(x,\ell,-1))\varphi_x(x,\ell,-1)}.
\end{align*}
Hence,
\begin{align*}
	\tilde w_{xx}(x,\ell,-1)=&\e^{-1}\ell^{-2}\theta_-'(r)\left[D(\varphi(x,\ell,-1))\varphi_x(x,\ell,-1)\right]_{xx}\int_{\ell/2}^{x}\frac{ds}{\left[D(\varphi(s,\ell,-1))\varphi_x(s,\ell,-1)\right]^2}\\
	&\qquad +\frac{\e^{-1}\ell^{-2}\theta_-'(r)\left[D(\varphi(x,\ell,-1))\varphi_x(x,\ell,-1)\right]_x}{\left[D(\varphi(s,\ell,-1))\varphi_x(s,\ell,-1)\right]^2}\\
	&\qquad -\frac{\e^{-1}\ell^{-2}\theta_-'(r)\left[D(\varphi_x(x,\ell,-1))\varphi_x(x,\ell,-1)\right]_x}{\left[D(\varphi_x(x,\ell,-1))\varphi_x(x,\ell,-1)\right]^2}\\
	=&\e^{-1}\ell^{-2}\theta_-'(r)\e^{-2}f'(\varphi(x,\ell,-1))\varphi_x(x,\ell,-1)\int_{\ell/2}^{x}\frac{ds}{\left[D(\varphi(s,\ell,-1))\varphi_x(s,\ell,-1)\right]^2}\\
	=&\e^{-2}f'(\varphi(x,\ell,-1))w(x,\ell,-1),
\end{align*}
where we differentiate \eqref{eq:stat} to obtain the formula for $\left[D(\varphi(x,\ell,-1))\varphi_x(x,\ell,-1)\right]_{xx}$.
As a consequence, the function $\tilde w$ satisfies the differential equation
$$\e^2\tilde w_{xx}(x,\ell,-1)=\frac{f'(\varphi(x,\ell,-1))}{D(\varphi(x,\ell,-1))}\tilde w(x,\ell,-1).$$
Notice that the function $\varphi(x,\ell,-1)$ is monotone increasing in $(0,\ell/2)$ and one can choose $\e>0$ sufficiently small such that
$f'(\varphi(x,\ell,-1))>0$ for $x\in(0,\ell/2-\e)$;
then, the coefficient $f'/D$ is positive and since $\tilde w(x,\ell,-1)\leq0$, for any $x\in[0,\ell/2]$, the maximum principle yields
$$|\tilde w(x,\ell,-1)|\leq\max\left\{|\tilde w(0,\ell,-1)|,|\tilde w(\ell/2-\e,\ell,-1)|\right\},$$
which implies \eqref{eq:claim-w}.
Thus, there exists a constant $C>0$ such that
$$|w(x,\ell,-1)|\leq C\e^{-1}\delta_-(r), \qquad \mbox{ for } x\in[0,\ell/2-\e].$$
To deduce the estimate for $w$ on the interval $[\ell/2-\e,\ell/2]$, notice that
$$\e^2\left(D(\varphi)\varphi_x\right)^2=2\left(G(\varphi)-\theta_-\right)$$
implies $\e|\varphi_x|\leq C_1$, for some $C_1>0$ independent on $\e$ and, since
 $$\varphi(x,\ell,-1)\geq\varphi(\ell/2-\e,\ell,-1),\qquad \mbox{ for any } x\in[\ell/2-\e,\ell/2],$$
 $\varphi(\ell/2,\ell,-1)=\frac{\alpha+\beta}{2}$ and $G\left(\frac{\alpha+\beta}{2}\right)>0$,
if $r>0$ is sufficiently small, there exists $C_2>0$ such that $\e^2\left(D(\varphi)\varphi_x\right)^2\geq C_2$.
Substituting in the formula of $w$, we conlude 
$$|w(x,\ell,-1)|\leq C\e^{-4}\theta_-(r)\e^2(\ell/2-x)\leq C\e^{-1}\theta_-(r), \qquad \mbox{ for } x\in[\ell/2-\e,\ell/2],$$
where we used Theorem \ref{thm:teta-lead}.
We proved the bounds for $w(x,\ell,-1)$ when $x\in[0,\ell/2]$, the remaining cases are similar.
\end{proof}

\section{Layer dynamics}\label{sec:layer}
The main goal of this section is to derive a system of ODEs describing the exponentially slow motion of the layers for the metastable states of \eqref{eq:D-model}-\eqref{eq:Neu}.
With this aim, let us rewrite equation \eqref{eq:D-model} as
\begin{equation}\label{eq:op-L}
	u_t=\mathcal{L}(u):=\e^2(D(u)u_x)_x-f(u).
\end{equation}
Following \cite{Carr-Pego}, our analysis of the metastable dynamics will only be valid when the layer positions are well separated and bounded away from the boundary points $a,b$;
for fixed $\rho>0$, we define
\begin{equation*}
	\Omega_\rho:=\bigl\{{\bm h}\in\mathbb{R}^N\, :\,a<h_1<\cdots<h_N<b,\quad
		 h_j-h_{j-1}>\varepsilon/\rho\mbox{ for } j=1,\dots,N+1\bigr\},
\end{equation*}
where $h_0:=2a-h_1$ and $h_{N+1}:=2b-h_N$ (because of the homogeneous Neumann boundary conditions \eqref{eq:Neu}).
By construction, if $\rho_1<\rho_2$, then  $\Omega_{\rho_1}\subset \Omega_{\rho_2}$.

The idea is to associate to any $\bm h\in\Omega_\rho$ a function $u^{\bm h}=u^{\bm h}(x)$ which approximates a metastable
state with $N$ transition points at $h_1,\dots,h_N$ by matching the steady states \eqref{eq:per-l}, whose existence and uniqueness has been proved in Proposition \ref{prop:ell}.
The collection of $u^{\bm h}$ determines a $N$-dimensional manifold; in order to describe the layer dynamics, we shall introduce a projection which permits to separate the solution into a component on the manifold and a corresponding remainder.
Let us start by constructing the function $u^{\bm h}$:
for $\bm h\in\Omega_\rho$ with $\rho<\rho_0$ (where $\rho_0$ is the constant appearing in Proposition \ref{prop:ell}), 
we define the function $u^{\bm h}$ with $N$ transition points at $h_1,\dots,h_N$
by matching together different steady states satisfying \eqref{eq:per-l}, using smooth cut-off functions.
Given $\chi:\mathbb{R}\rightarrow[0,1]$ a $C^\infty$ function with $\chi(x)=0$ for $x\leq-1$ and $\chi(x)=1$ for $x\geq1$,
set 
\begin{equation*}
	\chi^j(x):=\chi\left(\frac{x-h_j}\varepsilon\right) \qquad\textrm{and}\qquad
	\varphi^j(x):=\varphi\left(x-h_{j-1/2},h_j-h_{j-1},(-1)^j\right),
\end{equation*}
where
\begin{equation*}
	h_{j-1/2}:=\frac12(h_{j-1}+h_j), \qquad \qquad j=1,\dots,N+1,
\end{equation*}
that is, $h_{j-1/2}$ is the midpoint of the line segment connecting $h_{j-1}$ and $h_j$; notice that $h_{1/2}=a$, $h_{N+1/2}=b$.
Then, the function $u^{\bm h}$ is given by the convex combination
\begin{equation}\label{u^h(x)}
	u^{\bm h}:=\left(1-\chi^j\right)\varphi^j+\chi^j\varphi^{j+1} \qquad \textrm{in}\quad I_j:=[h_{j-1/2},h_{j+1/2}], \qquad j=1,\dots,N,
\end{equation}
and we define the base manifold
\begin{equation*}
	\mathcal{M}:=\{u^{\bm h} :\bm h\in\Omega_\rho\}.
\end{equation*} 
By definition, $u^{\bm h}$  is a smooth function of $x$ and $\bm h$ and enjoys the properties
\begin{equation}\label{eq:rompiballe}
	\begin{aligned}
	&u^{\bm h}(a) =\varphi(0,2h_1-2a,-1)<\frac{\alpha+\beta}{2}, \\
	&u^{\bm h}(h_{j-1/2}) =\varphi\left(0,h_j-h_{j-1},(-1)^j \right), \quad j=1,\dots,N+1, \\
	&u^{\bm h}(h_j) =\frac{\alpha+\beta}{2},
			 \qquad \mathcal{L}(u^{\bm h}(x))=0 \quad \textrm{for }|x-h_j|\geq\varepsilon, \quad j=1,\dots,N.
	\end{aligned}
\end{equation}
In what follows, we use the notation
\begin{equation*}
	u^{\bm h}_j:=\partial_{h_j} u^{\bm h},
\end{equation*}
and we denote the tangent space to $\mathcal{M}$ at $u^{\bm h}$ by $T\mathcal{M}(u^{\bm h})=\mbox{span}\{u^{\bm h}_j : j=1,\dots,N\}$. 
At this point, the natural idea would be to construct a tubular neighborhood of $\mathcal{M}$,
with coordinates $(\bm h,w)$ where $w$ is orthogonal to $T\mathcal{M}(u^{\bm h})$.
Since $\mathcal{M}$ is not invariant, there is higher flexibility in the construction of its neighborhood
and tubular co-ordinates near $\mathcal{M}$ can be defined using approximate tangent vectors to $\mathcal{M}$. 
For $j=1,\dots,N$, introduce the cutoff function $\gamma^j$ given by
\begin{equation*}
	\gamma^j(x):=\chi\left(\frac{x-h_{j-1/2}-\varepsilon}\varepsilon\right)\left[1-\chi\left(\frac{x-h_{j+1/2}+\varepsilon}\varepsilon\right)\right].
\end{equation*}   
Then, the {\it approximate tangent vectors} $\tau^{\bm h}_j$ are defined by
\begin{equation*}
	\tau^{\bm h}_j(x):=-\gamma^j(x)D(u^{\bm h}(x))u^{\bm h}_x(x).
\end{equation*}
By construction, $\tau^{\bm h}_j$ are smooth functions of $x$ and $\bm h$ and are such that
\begin{equation}\label{eq:prop-tau}
	\begin{aligned}
	\tau^{\bm h}_j(x)&=0,				&\quad \textrm{for}\quad &x\notin[h_{j-1/2},h_{j+1/2}],\\
	\tau^{\bm h}_j(x)&=-D(u^{\bm h})u^{\bm h}_x(x),	&\quad \textrm{for}\quad &x\in[h_{j-1/2}+2\varepsilon,h_{j+1/2}-2\varepsilon]. 
	\end{aligned}
\end{equation}
As we will see later, the choice of the tangent vectors $\tau^{\bm h}_j$ plays a crucial role in the derivation of the equations describing the layer dynamics.

The key idea in our strategy is to approximate the metastable states for \eqref{eq:op-L} by using the function $u^{\bm h}$; 
hence, let us use the decomposition $u=u^{\bm h}+v$ for the solution, where $u^{\bm h}\in\mathcal{M}$ and the remainder $v$ satisfies 
\begin{equation}\label{eq:ort-con}
	(v,\tau^{\bm h}_j)_{L^2(a,b)}=\int_a^b v\tau^{\bm h}_j\,dx=:\langle v, \tau^{\bm h}_j\rangle=0, \qquad \qquad \mbox{ for any } j=1,\dots,N.
\end{equation}
By using the relation $v=u-u^{\bm h}$ and differentiating with respect to $t$ the relations \eqref{eq:ort-con}, we obtain
\begin{equation}\label{eq:h-1}
	\sum_{k=1}^N\left(\langle u^{\bm h}_k,\tau^{\bm h}_j\rangle-\langle v,\tau^{\bm h}_{jk}\rangle\right)h'_k=\langle\mathcal{L}(u^{\bm h}+v),\tau^{\bm h}_j\rangle, 
	 \qquad \qquad j=1,\dots,N,
\end{equation}
where, as above, we use the notation  
\begin{equation*}
	\tau^{\bm h}_{ji}:=\partial_{h_i} \tau^{\bm h}_j.
\end{equation*}
Equation \eqref{eq:h-1} should be coupled and studied together with the equation for $v$ (see \cite{Carr-Pego}), 
obtained by differentiating $v=u-u^{\bm h}$, which reads as
$$v_t=\mathcal{L}(u^{\bm h}+v)-\sum_{k=1}^Nu^{\bm h}_kh'_k.$$
In such a way, one may derive a system of the form 
$$\bm h'=\mathcal{F}(\bm h,v)$$
for the layers.
Inspired by \cite{Carr-Pego}, our objective is to derive a reduced system that governs the motion of the layers.
To this end, we employ the expansion
$$\mathcal{F}(\bm h,v)=\mathcal{P}^*(\bm h)+\mathcal{R}(\bm h,v),$$
where the leading-order term $\mathcal{P}^*(\bm h)$ is dominant, while the remainder $\mathcal{R}(\bm h,v)$ is significantly smaller and can thus be considered negligible.
We begin by considering an initial datum $u_0(x)=u^{\bm h(0)}(x)+v_0(x)$ where $v_0$ is sufficiently small so that $u_0$ has a $N$-transition layer structure, 
i.e. $u_0$ satisfies \eqref{eq:ass-u0} and \eqref{eq:energy-ini}.
By applying Theorem \ref{thm:main} we conclude that the solution $u$ maintains this $N$-transition layer structure for an exponentially long time.
Our derivation of the reduced system is based on the assumption that the terms involving $v$ are negligible in comparison to the leading-order contributions.
Consequently, by using the estimates in Section \ref{sec:est-per}, we derive an equation of the form $\bm h'=\mathcal{P}^*(\bm h)$ and in the next section we provide numerical evidence supporting the claim that this reduced system accurately describes the metastable dynamics for \eqref{eq:op-L}.
Following the aforementioned approach and by formally neglecting the terms involving $v$, we end up with 
\begin{equation}\label{eq:h-2}
	\sum_{k=1}^N\langle u^{\bm h}_k,\tau^{\bm h}_j\rangle h'_k=\langle\mathcal{L}(u^{\bm h}),\tau^{\bm h}_j\rangle, 
	 \qquad \qquad j=1,\dots,N.
\end{equation}
Hence, we need to find the leading-order terms coming from $\langle u^{\bm h}_k,\tau^{\bm h}_j\rangle$ and $\langle\mathcal{L}(u^{\bm h}),\tau^{\bm h}_j\rangle$.
To begin with, we consider the term $\langle\mathcal{L}(u^{\bm h}),\tau^{\bm h}_j\rangle$; 
for $j=1,\dots,N+1$, we set
\begin{equation*}
	r_{j}:=\frac{\varepsilon}{h_j-h_{j-1}},
\end{equation*}
and
\begin{equation*}
	\theta^{j}:=\left\{\begin{aligned}
		&\theta_-(r_{j})  	&j \textrm{ odd},\\
		&\theta_+(r_{j}) 	&j \textrm{ even},
		\end{aligned}\right.
	\qquad\qquad
	\delta^{j}:=\left\{\begin{aligned}
		&\delta_-(r_{j})	&j \textrm{ odd}, \\
		&\delta_+(r_{j})	&j \textrm{ even},
		\end{aligned}\right.
\end{equation*}
where $\theta_\pm$ and $\delta_\pm$ are defined in \eqref{eq:teta-delta}.
Thanks to these definitions and the suitable choice of $\tau^{\bm h}_j$, we can prove the following result.
\begin{lem}
Assume that $D,f$ satisfy \eqref{eq:ass-D}-\eqref{eq:ass-f}-\eqref{eq:ass-int0}-\eqref{eq:ass-int1}.
If $\rho>0$ is sufficiently small and $\bm h\in\Omega_\rho$, then 
\begin{equation}\label{eq:L-tang}
	\bigl\langle\mathcal{L}(u^{\bm h}),\tau^{\bm h}_j\bigr\rangle=\theta^{j+1}-\theta^j,
\end{equation} 
for any $j=1,\dots,N$.
\end{lem}
\begin{proof}
By using \eqref{eq:rompiballe} and, in particular that
\begin{equation*}
	\mathcal{L}(u^{\bm h}(x))=0,\qquad \textrm{for }|x-h_j|\geq\varepsilon,
\end{equation*}
we deduce
\begin{align*}
	\bigl\langle\mathcal{L}(u^{\bm h}),\tau^{\bm h}_j\bigr\rangle
		&=-\int_{h_j-\varepsilon}^{h_j+\varepsilon}\bigl\{\e^2(D(u^{\bm h})u^{\bm h}_x)_x-f(u^{\bm h})\bigr\}D(u^{\bm h})u^{\bm h}_x\,dx\\
		&=\left\{G(u^{\bm h})-\frac{\e^2}2\left(D(u^{\bm h})u^{\bm h}_x\right)^2\right\}\Bigg|_{h_j-\varepsilon}^{h_j+\varepsilon}=G(u^{\bm h}(h_{j+1/2}))-G(u^{\bm h}(h_{j-1/2}))\\
		&=\varphi\left(0,h_{j+1}-h_j,(-1)^{j+1}\right)-\varphi\left(0,h_{j}-h_{j-1},(-1)^{j}\right),
\end{align*}
and \eqref{eq:L-tang} follows from the definition of $\theta^j$.
\end{proof}
By substituting \eqref{eq:L-tang} in \eqref{eq:h-2} and using Theorem \ref{thm:teta-lead}, we obtain 
\begin{equation}\label{eq:h-3}
	\sum_{k=1}^N\langle u^{\bm h}_k,\tau^{\bm h}_j\rangle h'_k=\theta^{j+1}-\theta^j, 
	 \qquad \qquad j=1,\dots,N,
\end{equation}
with
\begin{equation*}
	\theta^{j}:=\left\{\begin{aligned}
	&\tfrac{1}{2}A_\alpha^2D(\alpha)^2K_\alpha^2\exp\left(-\frac{A_\alpha\ell_j}{\e}\right)\left(1+\mathcal{O}\left\{\e^{-1}\exp\left(-\frac{A_\alpha\ell_j}{2\e}\right)\right\}\right),
		 &j \textrm{ odd}, \\
		&\tfrac{1}{2}A_\beta^2D(\beta)^2K_\beta^2\exp\left(-\frac{A_\beta\ell_j}{\e}\right)\left(1+\mathcal{O}\left\{\e^{-1}\exp\left(-\frac{A_\beta\ell_j}{2\e}\right)\right\}\right), 	
		\qquad &j \textrm{ even},
		\end{aligned}\right.
\end{equation*}
where $\ell_j:=h_j-h_{j-1}$.
It is very important to notice that for some $C>0$
$$|\theta^j|\leq C\exp\left(-\frac{A\ell_j}{\e}\right), \qquad \qquad j=1,\dots,N+1,$$
where $A$ is defined in \eqref{eq:E-int}.
Hence, the right hand side of \eqref{eq:h-3} is exponentially small if the distance between the layers (the transition width) is greater than $\e$.
In order to provide a formula for the left hand side of \eqref{eq:h-3}, we need some auxiliary results.
First, we consider the functions $u^{\bm h}_j$, that is the derivative with respect to $h_j$ of the function $u^{\bm h}$ defined in \eqref{u^h(x)}.

\begin{lem}\label{lem:u^h_j}
Assume that $D,f$ satisfy \eqref{eq:ass-D}-\eqref{eq:ass-f} and \eqref{eq:ass-int0}-\eqref{eq:ass-int1}.
If $\rho>0$ is sufficiently small and $\bm h\in\Omega_\rho$, then the support of $u^{\bm h}_j$ is contained in the interval $[h_{j-1}-\e,h_{j+1}+\e]$ and 
\begin{equation*}
	u^{\bm h}_j=
	\begin{cases}
		\chi^{j-1}w^j, \qquad \qquad \qquad & \textrm{ in } I_{j-1},\\
		(1-\chi^j)(-\varphi_x^j+w^j)-\chi^j(\varphi_x^{j+1}+w^{j+1})+\chi^j_x(\varphi^j-\varphi^{j+1}), &\mbox{ in } I_{j},\\
		-(1-\chi^{j+1})w^{j+1}, &\mbox{ in } I_{j+1}, 
	\end{cases}
\end{equation*}
where $\chi^j, I_j$ are the same as in \eqref{u^h(x)} and 
$$w^j(x):=w(x-h_{j-1/2},h_j-h_{j-1},(-1)^j),$$
with $w$ defined in Lemma \ref{lem:varphi_l}.
\end{lem}
\begin{proof}
Consider the definition \eqref{u^h(x)} of $u^{\bm h}$ and notice that only the functions $\chi^j, \varphi^j$ and $\varphi^{j+1}$ depend on $h_j$. 
As a consequence, we need to focus the attention only on $I_{j-1}, I_j$ and $I_{j+1}$ when differentiating with respect to $h_j$.
In $I_{j-1}$ one has 
\begin{align*}
	u^{\bm h}_j(x)&=\chi^{j-1}(x)\frac{\partial}{\partial h_j}\varphi^j(x)\\
	&=\chi^{j-1}(x)\left[-\tfrac12\varphi_x\left(x-h_{j-1/2},h_j-h_{j-1},(-1)^j\right)+\varphi_\ell\left(x-h_{j-1/2},h_j-h_{j-1},(-1)^j\right)\right]\\
	&=\chi^{j-1}(x)w^{j}(x),
\end{align*}
where in the last passage we used Lemma \ref{lem:varphi_l} and the fact that $x-h_{j-1/2}<0$ in $I_{j-1}$.
Moreover, since $\chi^{j-1}(x)=0$ if $x\leq h_{j-1}-\e$, we deduce 
$$u^{\bm h}_j(x)=0, \qquad \mbox{ for }x\in[h_{j-3/2},h_{j-1}-\e].$$
Similarly, in $I_j$ \eqref{u^h(x)} implies
\begin{align*}
	u^{\bm h}_j&=(1-\chi^{j})\frac{\partial}{\partial h_j}\varphi^j+\chi^j\frac{\partial}{\partial h_j}\varphi^{j+1}+\frac{\partial}{\partial h_j}\chi^{j}(\varphi^{j+1}-\varphi^j)\\
	&=(1-\chi^{j})(-\varphi^j_x+w^j)+\chi^j(-\varphi^{j+1}_x-w^{j+1})-\chi^{j}_x(\varphi^{j+1}-\varphi^j),
\end{align*}
where we used Lemma \ref{lem:varphi_l} and the fact that $x-h_{j-1/2}>0$ and $x-h_{j+1/2}<0$ in $I_{j}$.
Finally, in $I_{j+1}$ it holds 
\begin{align*}
	\frac{\partial}{\partial h_j}\varphi^{j+1}(x)&=-\tfrac12\varphi_x\left(x-h_{j+1/2},h_{j+1}-h_j,(-1)^{j+1}\right)\\
	&\qquad \qquad-\varphi_\ell\left(x-h_{j+1/2},h_{j+1}-h_j,(-1)^{j+1}\right)\\
	&=-w^{j+1}(x)
\end{align*}
because of Lemma \ref{lem:varphi_l} and the fact that $x-h_{j+1/2}>0$ in $I_{j+1}$.
Hence, we end up with
\begin{equation*}
	u^{\bm h}_j(x)=(1-\chi^{j+1}(x))\frac{\partial}{\partial h_j}\varphi^{j+1}(x)=-(1-\chi^{j+1}(x))w^{j+1}(x), \qquad \quad x\in I_{j+1}.
\end{equation*}
In conclusion, we obtain the formula for $u^{\bm h}_j$ and since $\chi^{j+1}(x)=1$ if $x\geq h_{j+1}+\e$,
the interval $[h_{j-1}-\e,h_{j+1}+\e]$ contains its support.
\end{proof}
From the formula for $u^{\bm h}_j$ presented in Lemma \ref{lem:u^h_j}, it follows that we need upper bounds for $|\varphi^j-\varphi^{j+1}|$ and $|\varphi^j_x-\varphi^{j+1}_x|$.
\begin{lem}\label{lem:varphi^j-diff}
Assume that $D,f$ satisfy \eqref{eq:ass-D}-\eqref{eq:ass-f} and \eqref{eq:ass-int0}-\eqref{eq:ass-int1}.
If $\rho>0$ is sufficiently small and $\bm h\in\Omega_\rho$, then there exists $\e_0,C_1,C_2>0$ such that 
for $\e\in(0,\e_0)$ it holds 
\begin{align*}
	|\varphi^j(x)-\varphi^{j+1}(x)|&\leq C_1|\theta^j-\theta^{j+1}|,\\
	|\varphi^j_x(x)-\varphi^{j+1}_x(x)|&\leq C_2\e^{-1}|\theta^j-\theta^{j+1}|,
\end{align*}
for any $x\in[h_j-\e,h_j+\e]$.
\end{lem}
\begin{proof}
Set $\nu(x):=\varphi^j(x)-\varphi^{j+1}(x)$ and notice that $\nu(h_j)=0$.
In view of \eqref{eq:phipm-first}, we deduce
\begin{align*}
	\e\nu'&=\sqrt2\left(\frac{\sqrt{G(\varphi^j)-\theta^j}}{D(\varphi^j)}-\frac{\sqrt{G(\varphi^{j+1})-\theta^{j+1}}}{D(\varphi^{j+1})}\right)\\
	&=\frac{\sqrt2}{D(\varphi^j)}\left(\sqrt{G(\varphi^j)-\theta^j}-\sqrt{G(\varphi^{j+1})-\theta^{j+1}}\right)\\
	&\qquad\quad+\sqrt2\left(\frac{1}{D(\varphi^{j})}-\frac{1}{D(\varphi^{j+1})}\right)\sqrt{G(\varphi^{j+1})-\theta^{j+1}}\\
	&=\frac{\sqrt2}{D(\varphi^j)\left(\sqrt{G(\varphi^j)-\theta^j}+\sqrt{G(\varphi^{j+1})-\theta^{j+1}}\right)}\left[G(\varphi^j)-G(\varphi^{j+1})+\theta^{j+1}-\theta^j\right]\\
	&\qquad\quad+\frac{\sqrt2}{D(\varphi^j)D(\varphi^{j+1})}\left(D(\varphi^{j+1})-D(\varphi^{j})\right)\sqrt{G(\varphi^{j+1})-\theta^{j+1}}.
\end{align*}
As a consequence, there exists a positive constant $C>0$ such that
$$\e|\nu'|\leq C\left(|\varphi^j-\varphi^{j+1}|+|\theta^{j+1}-\theta^j|\right)=C\left(|\nu|+|\theta^{j+1}-\theta^j|\right).$$
By using Gronwall's inequality, we end up with
$$|\nu|\leq C_1|\theta^{j+1}-\theta^j|,$$
for some $C_1>0$. 
Hence,
$$\e|\nu'|\leq C_2|\theta^{j+1}-\theta^j|,$$
for some $C_2>0$ and the proof is complete.
\end{proof}
Now we are ready to establish a formula for the left hand side of \eqref{eq:h-3}.
\begin{thm}\label{thm:matrix-S}
Assume that $D,f$ satisfies \eqref{eq:ass-D}-\eqref{eq:ass-f} and \eqref{eq:ass-int0}-\eqref{eq:ass-int1}.
There exists $\bar\rho,C>0$ such that if $\rho\in(0,\bar\rho)$ and $\bm h\in\Omega_\rho$, then
\begin{equation}\label{eq:S-diag}
	|\langle u^{\bm h}_j,\tau^{\bm h}_j\rangle-\e^{-1}S_G|\leq C\e^{-1}\max\left\{\delta^j,\delta^{j+1}\right\}\leq C\e^{-1}\exp\left(-\frac{A\ell_j}{2\e}\right),
	\qquad j=1,\dots,N,	
\end{equation}
where $S_G$ and $A$ are defined in \eqref{eq:S_G} and \eqref{eq:E-int}, respectively.
Moreover,
\begin{equation}\label{eq:S-tridiag}
	|\langle u^{\bm h}_j,\tau^{\bm h}_{j+1}\rangle|+|\langle u^{\bm h}_{j+1},\tau^{\bm h}_j\rangle|\leq C\e^{-1}\delta^{j+1}\leq C\e^{-1}\exp\left(-\frac{A\ell_{j+1}}{2\e}\right),
	\qquad j=1,\dots,N-1,	
\end{equation}
and 
\begin{equation}\label{eq:S-null}
	\langle u^{\bm h}_k,\tau^{\bm h}_j\rangle=0 \; \mbox{ if } \; |k-j|>1.
\end{equation}
\end{thm}
\begin{proof}
In view of \eqref{eq:prop-tau}, $\tau^{\bm h}_j=0$ out of $I_j=[h_{j-1/2},h_{j+1/2}]$;
on the other hand, from Lemma \ref{lem:u^h_j}, $u^{\bm h}_k=0$ out of $[h_{k-1}-\e,h_{k+1}+\e]$.
Hence, if $|k-j|>1$, then $I_j\cap[h_{k-1}-\e,h_{k+1}+\e]=\emptyset$ and  \eqref{eq:S-null} holds true.
Next, by using again \eqref{eq:prop-tau}, we infer
$$\langle u^{\bm h}_j,\tau^{\bm h}_{j+1}\rangle=\int_{I_{j+1}}u^{\bm h}_j\tau^{\bm h}_{j+1}\,dx.$$
From Lemma \ref{lem:u^h_j} and Lemma \ref{lem:est-w}, it follows that
$$|u^{\bm h}_j(x)|\leq C|w^{j+1}(x)|\leq C\e^{-1}\delta^{j+1}, \qquad \qquad \mbox{ for any } x\in I_{j+1}.$$
As a consequence,
$$|\langle u^{\bm h}_j,\tau^{\bm h}_{j+1}\rangle|\leq C\e^{-1}\delta^{j+1}\int_{I_{j+1}}|\tau^{\bm h}_{j+1}|\,dx\leq C\e^{-1}\delta^{j+1}\int_{I_{j+1}}|u^{\bm h}_x|\,dx,$$
where in the last inequality we used the definition of $\tau^{\bm h}_j$.
Notice that $u^{\bm h}_x$ is of one sign in $I_{j+1}$ and so, we have
$$|\langle u^{\bm h}_j,\tau^{\bm h}_{j+1}\rangle|\leq C\e^{-1}\delta^{j+1}.$$
Similarly, one can obtain the bound for $|\langle u^{\bm h}_{j+1},\tau^{\bm h}_j\rangle|$ and using Theorem \ref{thm:teta-lead} we end up with \eqref{eq:S-tridiag}.
It remains to prove \eqref{eq:S-diag}, which is the most interesting case.
In view of \eqref{eq:prop-tau}, we have
$$\langle u^{\bm h}_j,\tau^{\bm h}_{j}\rangle=\int_{I_j}u^{\bm h}_j\tau^{\bm h}_j\,dx.$$
By using the formula in Lemma \ref{lem:varphi_l}, let us rewrite for $x\in I_j$, $u^{\bm h}_j=\nu_1+\nu_2$ with
$$\nu_1:=-(1-\chi^j)\varphi_x^j-\chi^j\varphi_x^{j+1}, \qquad \nu_2:=(1-\chi^j)w^j-\chi^jw^{j+1}+\chi^j_x(\varphi^j-\varphi^{j+1}).$$
Also, let us rewrite for $x\in I_j$, $\tau^{\bm h}_j=D(u^{\bm h})(\nu_1+\nu_3)$ with
$$\nu_3:=-\gamma^ju^{\bm h}_x+(1-\chi^j)\varphi_x^j+\chi^j\varphi_x^{j+1}=(1-\gamma^j)u^{\bm h}_x+\chi^j_x(\varphi^j-\varphi^{j+1}),$$
where in the last passage we differentiate \eqref{u^h(x)}.
Thus,
$$\langle u^{\bm h}_j,\tau^{\bm h}_{j}\rangle=\int_{I_j}D(u^{\bm h})\nu_1^2\,dx+\int_{I_j}D(u^{\bm h})\nu_1(\nu_2+\nu_3)\,dx+\int_{I_j}D(u^{\bm h})\nu_2\nu_3\,dx.$$
For the first integral, taking advantage of the properties of $\chi^j$, we deduce 
\begin{equation*}
	\int_{I_j}D(u^{\bm h})\nu_1^2\,dx=\int_{h_{j-1/2}}^{h_j-\e}D(\varphi^j)(\varphi^j_x)^2\,dx+\int_{h_j+\e}^{h_{j+1/2}}D(\varphi^{j+1})(\varphi^{j+1}_x)^2\,dx+\mathcal{E},
\end{equation*}
where 
\begin{equation*}
	\mathcal{E}:=\int_{h_j-\e}^{h_j+\e}D(u^{\bm h})\nu_1^2\,dx.
\end{equation*}
Since we can write 
$$u^{\bm h}=\varphi^j+\chi^j(\varphi^{j+1}-\varphi^j),\qquad \nu_1=-\varphi_x^j+\chi^j(\varphi_x^j-\varphi_x^{j+1}), \qquad \qquad \mbox{ in } [h_j-\e,h_j],$$
and 
$$u^{\bm h}=\varphi^{j+1}+(1-\chi^j)(\varphi^j-\varphi^{j+1}),\qquad\nu_1=(1-\chi^j)(\varphi_x^{j+1}-\varphi_x^{j})-\varphi_x^{j+1}, \quad  \mbox{ in } [h_j,h_j+\e],$$
we obtain
\begin{equation*}
	\mathcal{E}=\int_{h_j-\e}^{h_j}D(\varphi^j)(\varphi_x^j)^2\,dx+\int_{h_j}^{h_j+\e}D(\varphi^{j+1})(\varphi_x^{j+1})^2\,dx+\mathcal{E}_1,
\end{equation*}
where (thanks to Lemma \ref{lem:varphi^j-diff}) $|\mathcal{E}_1|\leq C\e^{-1}|\theta^j-\theta^{j+1}|\leq C\e^{-1}\max\left\{\theta^j,\theta^{j+1}\right\}$.
Hence, 
\begin{equation*}
	\int_{I_j}D(u^{\bm h})\nu_1^2\,dx=\int_{h_{j-1/2}}^{h_j}D(\varphi^j)(\varphi^j_x)^2\,dx+\int_{h_j}^{h_{j+1/2}}D(\varphi^{j+1})(\varphi^{j+1}_x)^2\,dx+\mathcal{E}_1.
\end{equation*}
By changing variable, we obtain the integral computed in \eqref{eq:int-Dphi'} and so,
\begin{equation}\label{eq:S_G-int1}
	\int_{I_j}D(u^{\bm h})\nu_1^2\,dx=\e^{-1}S_G+E+\mathcal{E}_1,
\end{equation}
where $|E|\leq C\e^{-1}\max\left\{\delta^j,\delta^{j+1}\right\}$.
Moreover, since $\int_{I_j}|\nu_1|\leq C$, from Lemma \ref{lem:est-w} and Lemma \ref{lem:varphi^j-diff} it follows that
$$|\nu_2|+|\nu_3|\leq C\e^{-1}\max\{\delta^j,\delta^{j+1}\},$$
and as a consequence
\begin{equation}\label{eq:S_G-int2}
	\left|\int_{I_j}D(u^{\bm h})\nu_1(\nu_2+\nu_3)\,dx\right|+\left|\int_{I_j}D(u^{\bm h})\nu_2\nu_3\,dx\right|\leq C\e^{-1}\max\{\delta^j,\delta^{j+1}\}.
\end{equation}
Combining \eqref{eq:S_G-int1}, \eqref{eq:S_G-int2} and using Theorem \ref{thm:teta-lead}, we obtain \eqref{eq:S-diag} and the proof is complete. 
\end{proof}
Thanks to Theorem \ref{thm:matrix-S}, we can rewrite the ODEs \eqref{eq:h-3} in the compact form 
\begin{equation}\label{eq:h-4}
	S(\bm h)\bm h'=\mathcal{P}(\bm h),
\end{equation}
with $\mathcal{P}_j(\bm h):=\theta^{j+1}-\theta^j$, and the $n\times n$ tridiagonal matrix $S$ satisfying 
\begin{equation*}
	S(\bm h)=\e^{-1}S_G\mathcal{I}_n+\mathcal{O}\left\{\e^{-1}\exp\left(-\frac{A\ell^{\bm h}}{2\e}\right)\right\},
\end{equation*}
where $\mathcal{I}_n$ is the identity matrix and $\ell^{\bm h}:=\min \ell_j$.
Hence, \eqref{eq:h-4} reads as
\begin{equation}\label{eq:h-5}
	\bm h'=S^{-1}(\bm h)\mathcal{P}(\bm h),
\end{equation}
where
$$S^{-1}(\bm h)=\frac{\e}{S_G}\mathcal{I}_n+\mathcal{O}\left\{\e^{-1}\exp\left(-\frac{A\ell^{\bm h}}{2\e}\right)\right\}.$$
In particular, if we introduce the function $\mathcal{P}^*(\bm h)=\frac{\e}{S_G}\mathcal{I}_n\mathcal{P}(\bm h)$ given by
$$\mathcal{P}_j^*(\bm h):=\frac{\e}{S_G}(\theta^{j+1}-\theta^j),\qquad \qquad j=1,\dots, N,$$
we end up with
$$\|S^{-1}(\bm h)\mathcal{P}(\bm h)-\mathcal{P}^*(\bm h)\|_\infty\leq C\e^{-1}\exp\left(-\frac{A\ell^{\bm h}}{2\e}\right)\|\mathcal{P}(\bm h)\|_\infty,$$
where $\|\mathcal{P}\|_\infty:=\max|\mathcal{P}_j|$.
Therefore, we can rewrite \eqref{eq:h-5} as
$$\bm h'=\mathcal{P}^*(\bm h)+\bar{\mathcal{E}}(\bm h),$$
with $\|\bar{\mathcal{E}}(\bm h)\|_\infty$ much smaller than $\|\mathcal{P}^*(\bm h)\|_\infty$, 
so that the magnitude and direction of $\bm h'$ is determined by $\mathcal{P}^*(\bm h)$.
In conclusion, by neglecting $\bar{\mathcal{E}}(\bm h)$ we derive the system
\begin{equation}\label{eq:h_j-final}
	h_j'=\frac{\e}{S_G}(\theta^{j+1}-\theta^j), \qquad \qquad j=1,\dots, N.
\end{equation}
As we mentioned in the Introduction, equations \eqref{eq:h_j-final} generalize the result in \cite{Carr-Pego},
and their system can be obtained from \eqref{eq:h_j-final} by choosing $D\equiv1$ and $\alpha=-1, \beta=1$ in \eqref{eq:ass-f} and \eqref{eq:ass-int0}-\eqref{eq:ass-int1}. 

\section{Numerical evidence}\label{sec:num}
In this section we explore the metastable dynamics of the problem \eqref{eq:D-model} subject to \eqref{eq:Neu} and \eqref{eq:initial};
the main goal is to provide a numerical confirmation that the ODE system \eqref{eq:h_j-intro} describes accurately the slow layer motion. 
We recall once again that if we choose $D\equiv1$ and $f$ satisfying \eqref{eq:ass-f}-\eqref{eq:ass-int0}-\eqref{eq:ass-int1} with $\alpha=-1$ and $\beta=1$,
we recover the same equation of \cite{Carr-Pego}; hence, in our numerical experiments we will consider a non constant diffusion function $D$, that is 
the \emph{Mullins diffusion} $D(u)$ defined in \eqref{eq:MullinsD} with $D_0=1$.
Moreover, we choose $f(u)=u^3-u$, which satisfies \eqref{eq:ass-f}-\eqref{eq:ass-int0}-\eqref{eq:ass-int1} with $\alpha=-1$ and $\beta=1$.
Notice that, in this case, a metastable state $u^{\bm h}$ defined as in \eqref{u^h(x)} is equal to zero at the transition points $h_j$.
The explicit formulas for the even function $D$ and the odd function $f$ imply that the function $G$ defined in \eqref{eq:G} is an even function given by
$$G(u)=\frac{u^2-1}{2}-\ln(1+u^2)+\ln2,$$
and, as a consequence, the set of parameters appearing in \eqref{eq:h_j-intro} is given by
\begin{equation}\label{eq:ode-par}
	\begin{aligned}
		S_G:=\int_{-1}^1&\sqrt{u^2-2\ln(1+u^2)+2\ln2-1}\,du, \qquad A:=\frac{\sqrt{G''(\pm1)}}{D(\pm1)}=\frac12,\\
		&\qquad \quad K:=2\exp\left\{\int_0^1\left(\frac{D(t)}{2\sqrt{2G(t)}}-\frac{1}{1-t}\right)\,dt\right\}.
	\end{aligned}
\end{equation}
Our goal is to compute the numerical solution of the PDE \eqref{eq:D-model} and to compare the layer motion with the numerical solution of the ODE
\begin{equation}\label{eq:h_j-approx}
	h_j'=\frac{\e K^2}{32S_G}\left[\exp\left(-\frac{h_{j+1}-h_j}{2\e}\right)-\exp\left(-\frac{h_j-h_{j-1}}{2\e}\right)\right], \qquad \qquad j=1,\dots, N,
\end{equation}
where $S_G$ and $K$ defined in \eqref{eq:ode-par}.
The right hand side of \eqref{eq:h_j-approx} is obtained by using the explicit values of the constants \eqref{eq:ode-par} and by neglecting the smaller reminder in the expansion
for $\theta^j$ \eqref{eq:theta^j}; as it was already mentioned in the Introduction, since the functions $D,f$ are symmetric, \eqref{eq:h_j-approx} implies that the two closest layers
move towards each other with approximately the same speed and the other $N-2$ points are essentially static.

\begin{figure}[t]
\centering
\includegraphics[width=0.65\textwidth]{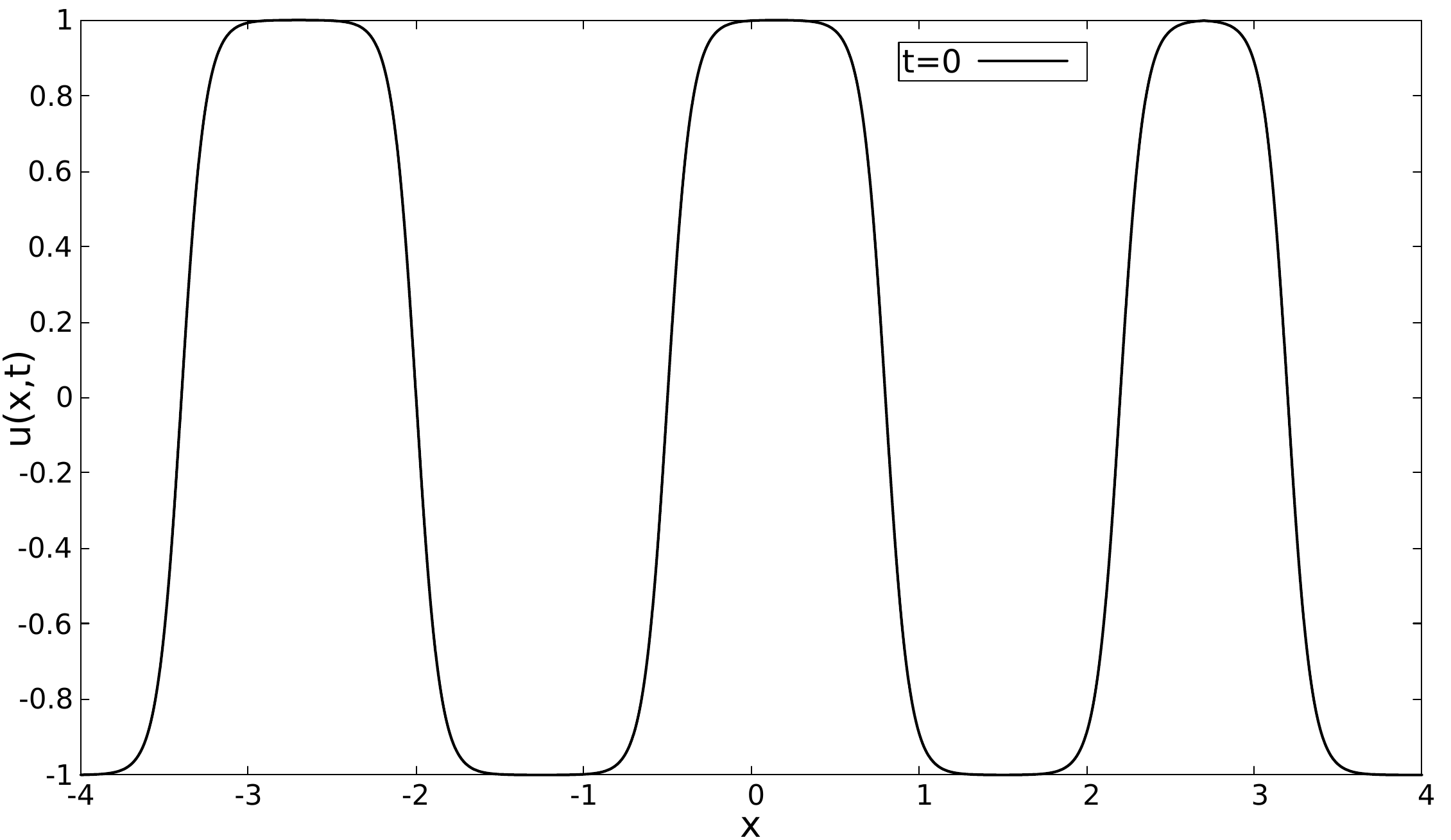} 
\caption{Initial datum $u_0$ with $N=6$ located at -3.4, -2, -0.5, 0.8, 2.2, 3.2. 
\label{fig:pde:icond}}
\end{figure}
\begin{figure}[t]
\centering
\includegraphics[scale = 0.5]{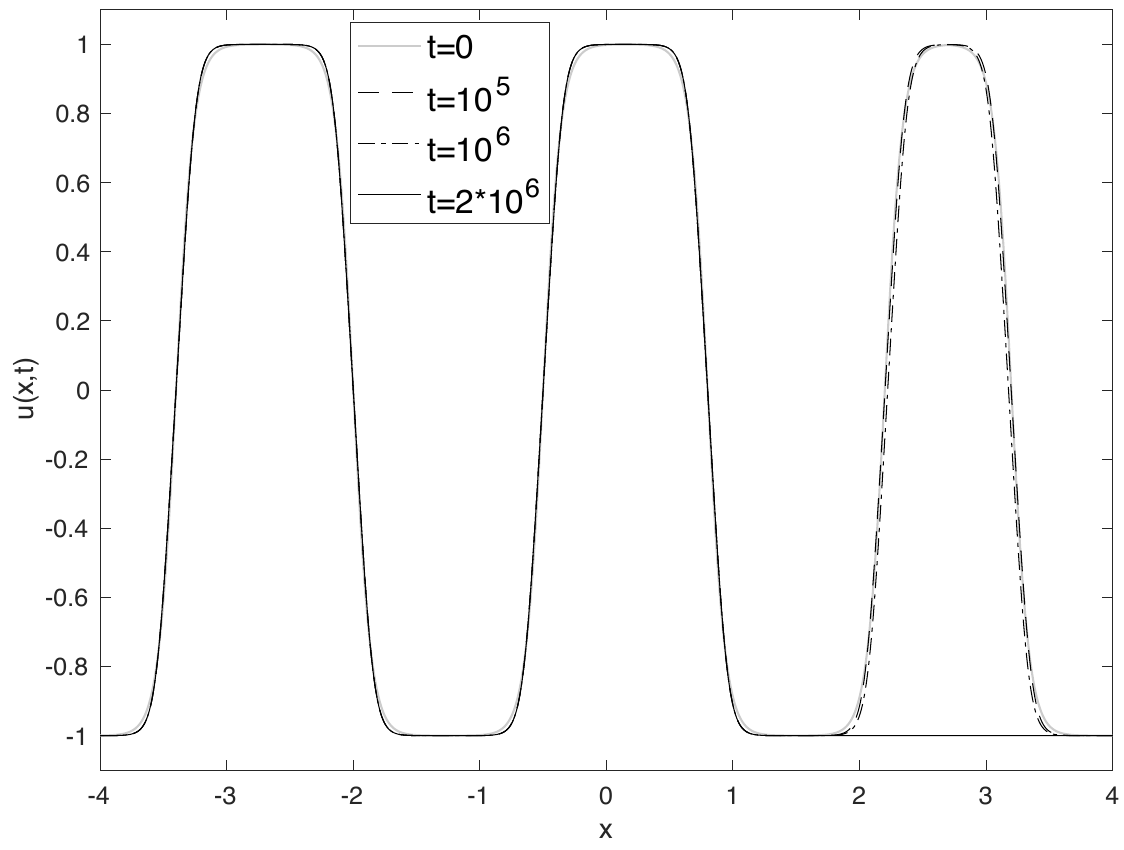}
\caption{Numerical solutions to \eqref{eq:D-model}-\eqref{eq:Neu}-\eqref{eq:initial} with $D(u)=(1+u^2)^{-1}$, $f(u)=u^3-u$ and $\e=0.1$.
The initial datum $u_0$ with 6 transitions is depicted in Figure \ref{fig:pde:icond}.
This Figure is taken from \cite[Figure 3]{FHLP}.}
\label{fig:Mullins}
\end{figure}

In the following experiments we set different values for $\e$ and different initial conditions $u_0$ with a $N$-transition layer structure in the spatial interval $[-4,4]$
as given by \cite[Eq. (3.5)]{FHLP}, with $\Phi_\e(x):=\tanh(x/\sqrt2\e)$,
which corresponds to $\Phi$ in Lemma \ref{lem:stand} for the linear case $D\equiv 1$, see Figure \ref{fig:pde:icond}.
We then solve numerically the ODE system \eqref{eq:h_j-approx} using the Matlab command \texttt{ode15s} which is a time-step adaptive scheme (see \cite{ShampineReiche}). 
Moreover, to verify this approximate solution, we also solve numerically the PDE problem \eqref{eq:D-model}
using continuous linear Lagrange Finite Elements (FEs) in space (see \cite{Ern.Guermond:04}) combined with the backward Euler method, 
also known as Backward Difference (BDF) method; for the time discretization (see \cite{Thomee}). 
The formal convergence analysis of this  FE-BDF numerical method for solving non-linear PDEs has been well documented in \cite{Wheeler}  and \cite[Ch. 13]{Thomee}. 
We use $N_h=4000$ finite elements, and set a constant time step $\Delta t=0.5$ for all numerical experiments.
The FE-BDF numerical scheme has been implemented in the  \texttt{C++} programming language along with  the linear algebra \texttt{Eigen} open-source library (see \url{http://eigen.tuxfamily.org}), and the linear direct solver \texttt{Intel MKL Pardiso} \cite{Schenk.Gartner.ea:01}.

\subsection*{Numerical experiment with $\pmb{\e=0.1}$}
We set $\e=0.1$ and an initial datum $u_0$ with a $6$-transition layer structure;
the Figure \ref{fig:pde:icond} illustrates this initial condition.
As has been shown in \cite{FHLP}, the evolution of the initial condition depicted in Figure \ref{fig:pde:icond} is very slow and 
the solution has the same transition layer structure as the initial datum at time $t=10^6$; 
in fact, it looks like the layers, that is the points where the solution is equal to zero, are static.
However, at time $t=2\times10^6$ the two closest layers have already collapsed and the solution has a $4$-transition layer structure, see Figure \ref{fig:Mullins}.

Now, we want compare the layers motion of the numerical solution with the one governed by \eqref{eq:h_j-approx}.
To do so, we compute the locations of the zeros of the numerical solution at different times and compare it with \eqref{eq:h_j-approx}.
The Table \ref{tbl:ntest1} shows the computed values for $d_j(t):=h_j(t)-h_j(0)$ for any $j\in\{1,\dots,6\}$ at the times $t\in\{1.0\times10^{4}, 3.0\times10^{5}, 6.0\times10^{5}, 9.0\times10^{5}, 1.2\times10^{6}\}$ using the numerical solutions for the ODE (obtained with \texttt{ode15s}), and the PDE (obtained with the FE-BDF numerical scheme).
We can observe an acceptable agreement for all values of $d_j(t)$ between both numerical solutions.
\begin{table}
\centering
\begin{adjustbox}{max width=1.0\textwidth,center}
  \begin{tabular}{cccccc}
    \toprule
    \midrule
        &
    $t=1\times10^{4}$ &
    $t=3\times10^{5}$ &
    $t=6\times10^{5}$ &
    $t=9\times10^{5}$ &
    $t=1.2\times10^{6}$ \\
    \toprule
    \midrule
    \multicolumn{1}{c}{}&
    \multicolumn{5}{c}{Solving the ODE \eqref{eq:h_j-approx}}\\
    \midrule
$d_1$ & $-2.6977 \times 10^{-6}$ & $-8.1064\times10^{-5}$ & $-1.6239\times10^{-4}$ &$-2.4400\times10^{ -4}$ & $-3.2588\times10^{-4}$ \\
$d_2$ &$-4.3518\times10^{-8}$ &$-1.3044\times10^{-6}$ &$-2.6064\times10^{-6}$&$-3.9061\times{10}^{-6}$  &  $-5.2034\times10^{-6}$\\
$d_3$ & $\phantom{-}3.6508\times10^{-7}$ & $\phantom{-}1.0954\times10^{-5}$ & $\phantom{-}2.1914\times10^{-5}$ & $\phantom{-}3.2878 \times10^{-5}$ & $\phantom{-}4.3847\times10^{-5}$ \\
$d_4$ &$-3.2164\times10^{-7}$ &$-9.7195\times10^{-6}$ &$-1.9594\times10^{-5}$ &$-2.9645 \times10^{-5}$ &$-3.9902\times10^{-5}$\\
$d_5$ & $\phantom{-}1.5043\times10^{-4}$ & $\phantom{-}4.9610\times10^{-3}$ & $\phantom{-}1.1157\times10^{-2}$ & $\phantom{-}1.9414\times10^{-2}$ & $\phantom{-}3.1830\times10^{-2}$\\
$d_6$ & $-1.5048\times10^{-4}$ & $-4.9624\times10^{-3}$ & $-1.1159\times10^{-2}$ & $-1.9418\times10^{-2}$ & $-3.1835\times10^{-2}$\\
 \midrule
    \multicolumn{1}{c}{}&
    \multicolumn{5}{c}{Solving the PDE \eqref{eq:D-model}}\\
    \midrule
$d_1$ &    $-2.6958\times10^{-6}$ &    $-8.0980\times10^{-5}$ &  $-1.6223\times10^{-4}$ &  $-2.4375\times10^{-4}$  &  $-3.2554\times10^{-4}$\\
$d_2$ & $-4.3499\times10^{-8}$ & $-1.3027\times10^{-6}$ & $-2.6031\times10^{-6}$ & $-3.9011\times10^{-6}$ & $-5.1967\times10^{-6}$\\
$d_3$ & $\phantom{-}3.6486\times10^{-7}$ & $\phantom{-}1.0942\times10^{-5}$ & $\phantom{-}2.1889\times10^{-5}$ & $\phantom{-}3.2840\times10^{-5}$ &  $\phantom{-}4.3796\times10^{-5}$\\
$d_4$ & $-3.2144\times10^{-7}$ & $-9.7083\times10^{-6}$ & $-1.9572\times10^{-5}$ & $-2.9610\times10^{-5}$ &   $-3.9855\times10^{-5}$\\
$d_5$ & $\phantom{-}1.5032\times10^{-4}$ & $\phantom{-}4.9566\times10^{-3}$ & $\phantom{-}1.1146\times10^{-2}$ &   $\phantom{-}1.9392\times10^{-2}$ & $\phantom{-}3.1781\times10^{-2}$\\
$d_6$ & $-1.5037\times10^{-4}$ & $-4.9581\times10^{-3}$ &  $-1.1149\times10^{-2}$ &  $-1.9396\times10^{-2}$ &  $-3.1786\times10^{-2}$\\
 \bottomrule
 \end{tabular}
\end{adjustbox} 
\vskip0.1cm
 \caption{Evolution of the distance traveled by any layer at different times by solving numerically the ODE system \eqref{eq:h_j-approx} and the PDE \eqref{eq:D-model}
with $D(u)=(1+u^2)^{-1}, f(u)=u^3-u$, and $\e=0.1$. The initial condition $u_0$ has $6$ transitions at $x\in\{-3.4,-2,-0.5,0.8,2.2,3.2\}$, see Figure \ref{fig:pde:icond}.
\label{tbl:ntest1}}
\end{table}

Finally, in Table \ref{tbl:dtest1} we compare the numerical time obtained by solving the ODE \eqref{eq:h_j-approx} and the PDE \eqref{eq:D-model},
denoted by $t(ODE)$ and $t(PDE)$, respectively, necessary for the distance between the two closest layers to be given by the values in the first column.

\begin{table}
\centering
\begin{adjustbox}{max width=1.0\textwidth,center}
  \begin{tabular}{ccc}
    \toprule
    \midrule
    $h_6-h_5$ &
    $t(ODE)$ &
    $t(PDE)$ \\
    \toprule
    \midrule
$0.99$ & $3.01839\times 10^{5}$ & $3.02206\times10^{5}$ \\
$0.95$ & $1.05257\times10^{6}$ &$1.05425\times10^{6}$ \\
$0.75$ & $1.65393\times10^{6}$ & $1.65624\times10^{6}$ \\
$0.5$ &$1.66507\times10^{6}$ & $1.66739\times10^{6}$\\
$0.25$ & $1.66515\times10^6$ & $1.66747\times10^6$\\
 \bottomrule
 \end{tabular}
\end{adjustbox} 
\vskip0.1cm
 \caption{Time necessary for the distance $h_6-h_5$ to be given by the values in the first column.
 The initial condition $u_0$ has $6$ transitions at $x\in\{-3.4,-2,-0.5,0.8,2.2,3.2\}$, see Figure \ref{fig:pde:icond}.}
\label{tbl:dtest1}
\end{table}

\subsection*{Numerical experiment with $\pmb{\varepsilon=0.07}$}
Here we set $\e=0.07$ and an initial datum $u_0$ with a 6-transition layer structure as in the first experiment, but we change the positions of the transitions,
which are now located at $x\in\{-3.4,-2,-0.5,0.8,2.2,2.9\}$.
However, given the smaller value of $\e$, it takes much longer time for $d_j(t)$ to be large enough (at least greater than $10^{-10}$);
therefore, we only track the evolution of the \emph{fastest} $h_j(t)$,
which in this particular case are $h_5(t)$ and $h_6(t)$ (since those transitions are the closest ones).
Hence, the Table \ref{tbl:ntest2} shows the computed values of both $d_5(t)$ and $d_6(t)$
at different times $t\in\{1\times10^{4}, 2\times10^{4}, 4\times10^{4}, 6\times10^{4}, 1.0\times10^{5}\}$ 
using the numerical solutions for the  ODE (obtained with \texttt{ode15s}), and the PDE (obtained with the FE-BDF numerical scheme).
Again, we can observe an acceptable agreement between both numerical solutions.

\begin{table}
  \centering
\begin{adjustbox}{max width=1.0\textwidth,center}
  \begin{tabular}{cccccc}
    \toprule
    \midrule
          &
    $t=1\times10^{4}$ &
    $t=2\times10^{4}$ &
    $t=4\times10^{4}$ &
    $t=6\times10^{4}$ &
    $t=1.0\times10^{5}$ \\
    \toprule
    \midrule
    \multicolumn{1}{c}{}&
    \multicolumn{5}{c}{Solving the ODE \eqref{eq:h_j-approx}}\\
    \midrule

      $d_5$ & $\phantom{-}1.0534\times10^{-4}$ & $\phantom{-}2.1132\times10^{-4}$ & $\phantom{-}4.2522\times10^{-4}$ & $\phantom{-}6.4177\times10^{-4}$ & $\phantom{-}1.0831\times10^{-3}$ \\
      $d_6$ & $-1.0534\times10^{-4}$ & $-2.1132\times10^{-4}$ & $-4.2522\times10^{-4}$ & $-6.4177\times10^{-4}$ & $-1.0831\times10^{-3}$ \\
  
 \midrule
    \multicolumn{1}{c}{}&
    \multicolumn{5}{c}{Solving the PDE \eqref{eq:D-model}}\\
    \midrule
$d_5$& $\phantom{-}1.0517\times10^{-4}$ &     $\phantom{-}2.1098\times10^{-4}$ &  $\phantom{-}4.2453\times10^{-4}$ & $\phantom{-}6.4071\times10^{-4}$ & $\phantom{-}1.0813\times10^{-3}$\\
$d_6$& $-1.0517\times10^{-4}$ &    $-2.1098\times10^{-4}$ & $-4.2453\times10^{-4}$ & $-6.4071\times10^{-4}$&  $-1.0813\times10^{-3}$\\
 
 \bottomrule
 \end{tabular}
\end{adjustbox} 
\vskip0.1cm
 \caption{Evolution of $d_5(t)$ and $d_6(t)$ at different times by solving numerically the ODE system \eqref{eq:h_j-approx} and the PDE \eqref{eq:D-model}
with $D(u)=(1+u^2)^{-1}, f(u)=u^3-u$, and $\e=0.07$. The initial condition $u_0$ has $6$ transitions at $x\in\{-3.4,-2,-0.5,0.8,2.2,2.9\}$.\label{tbl:ntest2}}
\end{table}

\subsection*{Numerical experiment with $\pmb{\varepsilon=0.05}$}
Finally, in this experiment we set $\e=0.05$ and an initial datum $u_0$ with a 4-transition layer structure with layers located at $x\in\{-2, -1.3, -0.6, 0\}$.
As in the previous experiment, given the smaller value of $\e$, it takes much longer time to all $h_j(t)$ to change position, 
therefore we only track the evolution of the \emph{fastest} $h_j(t)$, which are $h_3(t)$ and $h_4(t)$ in this case.
Then, the Table \ref{tbl:ntest3} shows the computed values of both $h_3(t)$ and $h_4(t)$ 
at different times $t\in\{1.0\times10^{4}, 2.0\times10^{4}, 4.0\times10^{4}, 6.0\times10^{4}, 1.0\times10^{5}\}$ using the numerical solutions for the  ODE (obtained with \texttt{ode15s}), and the PDE (obtained with the FE-BDF numerical scheme).
We again can observe an acceptable agreement between both numerical solutions.

\begin{table}
  \centering
\begin{adjustbox}{max width=1.0\textwidth,center}
  \begin{tabular}{cccccc}
    \toprule
    \midrule
          &
    $t=1\times10^{4}$ &
    $t=2\times10^{4}$ &
    $t=4\times10^{4}$ &
    $t=6\times10^{4}$ &
    $t=1.0\times10^{5}$ \\
    \toprule
    \midrule
    \multicolumn{1}{c}{}&
    \multicolumn{5}{c}{Solving the ODE \eqref{eq:h_j-approx}}\\
    \midrule
      $d_3$ & $\phantom{-}1.3489\times10^{-6}$ & $\phantom{-}2.6979\times10^{-6}$ & $\phantom{-}5.3965\times10^{-6}$ & $\phantom{-}8.0956\times10^{-6}$ & $\phantom{-}1.3496\times10^{-5}$ \\
      $d_4$ & $-1.3741\times10^{-6}$       & $-2.7483\times10^{-6}$       & $-5.4971\times10^{-6}$       & $-8.2466\times10^{-6}$       & $-1.3747\times10^{-5}$ \\
 
 \midrule
    \multicolumn{1}{c}{}&
    \multicolumn{5}{c}{Solving the PDE \eqref{eq:D-model}}\\
    \midrule
$d_3$ & $\phantom{-}1.3437\times10^{-6}$  & $\phantom{-}2.6871\times10^{-6}$ & $\phantom{-}5.3743\times10^{-6}$ & $\phantom{-}8.0621\times10^{-6}$ & $\phantom{-}1.3440\times10^{-5}$\\
$d_4$ & $-1.3687\times10^{-6}$& $-2.7371\times10^{-6}$ & $-5.4744\times10^{-6}$ &  $-8.2123\times10^{-6}$ &   $-1.3690\times10^{-5}$\\
 \bottomrule
 \end{tabular}
\end{adjustbox} 
\vskip0.1cm
 \caption{Evolution of $d_3(t)$ and $d_4(t)$ at different times by solving numerically the ODE system \eqref{eq:h_j-intro} and the PDE \eqref{eq:D-model}
with $D(u)=(1+u^2)^{-1}, f(u)=u^3-u$, and $\e=0.05$. The initial condition $u_0$ has $4$ transitions at $x\in\{-2, -1.3, -0.6, 0\}$.\label{tbl:ntest3}}
\end{table}

All the numerical experiments presented show that the ODE system \eqref{eq:h_j-approx} 
accurately describes the metastable dynamics of the solutions to \eqref{eq:D-model}-\eqref{eq:Neu}.  
      
\section*{Acknowledgements}
The authors thank an anonymous referee for useful observations that improved the paper.
The work of R. Folino was partially supported by DGAPA-UNAM, program PAPIIT, grant IN-103425. L.F. L\'opez R\'ios was partially supported by DGAPA-UNAM, program PAPIIT, grant IN113225, and SECIHTI grant CF-2023-G-122. J. A. Butanda Mej\'ia was supported by SECIHTI ``Estancias Posdoctorales por M\'exico 2022''.

\end{document}